\documentclass[12pt]{article}

 \usepackage[body={6.5in,9.0in},left=1.0in, top=1.0in,centering]{geometry}                

\usepackage{graphicx,subfigure,amsmath,amsthm}
\usepackage{amssymb,mathrsfs}
\usepackage{epstopdf}
\usepackage[usenames, dvipsnames]{color}
\usepackage{pstricks}
\usepackage[letterpaper=true,colorlinks=true,urlcolor=blue,citecolor=blue,linkcolor=blue,pdfstartview=FitH]{hyperref}

\usepackage{lipsum}

\usepackage[sort,comma]{natbib}
\allowdisplaybreaks
\makeatletter \@addtoreset{equation}{section}

\makeatletter
\newcommand{\pushright}[1]{\ifmeasuring@#1\else\omit\hfill$\displaystyle#1$\fi\ignorespaces}
\newcommand{\pushleft}[1]{\ifmeasuring@#1\else\omit$\displaystyle#1$\hfill\fi\ignorespaces}
\makeatother

\def\F{{\mathcal F}}
\def\G{\mathcal G}
 
\def\ss{{\mathbb  S}}
\def\lan{\big\langle}
\def\ran{\big\rangle}
\newcommand{\LL}{\mathcal L}
\newcommand{\R}{\mathbb R}

\newcommand{\B}{\mathfrak B}
\renewcommand{\d}{\mathrm{d}}

\newcommand{\la}{\lambda}

\newcommand{\e}{\varepsilon}

 \newcommand{\lbar }{\overline}

\newcommand{\sg}{\sigma}
\newcommand{\sgla}{\sigma_{\lambda_{0}}}
 \newcommand{\tr}{\text {tr}}
 \newcommand{\var}{\text{Var}}
\renewcommand{\P}{\mathbb P}
\newcommand{\Q}{{\mathbb Q}}

\newcommand{\E}{{\mathbb E}}

\newcommand{\N}{\mathbb N}

\newcommand{\wdt}{\widetilde}
\newcommand{\wdh}{\widehat}

\newcommand{\cadlag}{c\`adl\`ag }

\newtheorem{thm}{Theorem}[section]
\newtheorem{prop}[thm]{Proposition}
\newtheorem{lem}[thm]{Lemma}
\newtheorem{cor}[thm]{Corollary}

\theoremstyle{definition}
\newtheorem{Assumption}[thm]{Assumption}
\newtheorem{rem}[thm]{Remark}

\newtheorem{exm}[thm]{Example}

\newcommand{\set}[1]{\left\{#1\right\}}

\newcommand{\norm}[1]{\left\| #1\right\|}
\def\fubao#1 {\fbox {\footnote {\ }}\ \footnotetext {From Fubao: #1}}
\def\chao#1 {\fbox {\footnote {\ }}\ \footnotetext {From Chao: #1}}

\title{\Large Jump Type Stochastic Differential Equations with Non-Lipschitz Coefficients: Non Confluence,  Feller and Strong Feller Properties, and Exponential Ergodicity}

\author{Fubao Xi\thanks{School of Mathematics and Statistics, Beijing Institute of Technology, Beijing 100081, China,
xifb@bit.edu.cn.}
\and Chao Zhu\thanks{Department of Mathematical Sciences, University of Wisconsin-Milwaukee, Milwaukee, WI 53201, USA, zhu@uwm.edu.}}

\begin{document}
\maketitle

\begin{abstract}
This paper considers multidimensional jump type stochastic differential equations with super linear  and  non-Lipschitz coefficients. After establishing a sufficient condition for nonexplosion, this paper  presents sufficient local non-Lipschitz conditions for pathwise uniqueness. The non confluence property for solutions is investigated.  Feller and strong Feller properties under local non-Lipschitz conditions are investigated via the coupling method. Sufficient conditions for irreducibility and exponential ergodicity are derived. As applications, this paper also studies multidimensional stochastic differential equations driven by L\'evy processes and presents a Feynman-Kac formula for L\'evy type operators.


\bigskip


\bigskip

\noindent{\bf Running Title.} Stochastic Differential Equations with Non-Lipschitz Coefficients
\bigskip

\noindent{\bf 2000 MR Subject Classification.} 60J25, 60J27, 60J60,
60J75.
\end{abstract}

\section{Introduction}\label{sect-introduction}

Let  $(U, \mathfrak U)$  be a measurable space and $\nu$ a $\sigma$-finite measure on $U$.
 Let  $d \ge 2 $ be a positive integer,  $b: \R^{d}\mapsto \R^{d}$, $\sigma:\R^{d}\mapsto \R^{d\times d}$ and $c: \R^{d}\times U \mapsto \R^{d}$ be Borel measurable functions.
Consider the following stochastic differential equation (SDE)
\begin{equation}
\label{eq-sde-sec-2}
\d X(t) = b(X(t)) \d t + \sigma(X(t))\d W(t) + \int_{U} c(X(t-),u)\wdt N(\d t, \d u),
\end{equation}
where  $W$ is a standard $d$-dimensional Brownian motion, and $N$ is a Poisson random measure on $[0,\infty)\times U$ with intensity $\d t \,\nu(\d u)$ and compensated Poisson random measure $\wdt N$.   It is well-known that if the coefficients $b,\sigma$ and $c$ of \eqref{eq-sde-sec-2} satisfy the linear growth and local Lipschitz conditions, then \eqref{eq-sde-sec-2} admits a non-exploding strong solution and the solution is pathwise unique; see, for example, \cite[Theorem IV.9.1]{IkedaW-89} for details.

The linear growth condition is a standard assumption in the literature; it   guarantees that
 the solution $X$  to   \eqref{eq-sde-sec-2} does not {\em explode} in finite time with probability one. But such a condition is often too restrictive in practice. For example, in many mathematical ecological  models (such as those in \cite{K-K-01,Mao-M-R-02,ZY-09a}),  the coefficients do not satisfy the linear growth condition; yet non-explosion is still guaranteed thanks to  the special structures of the underlying SDEs in these papers.  For general multidimensional SDEs without jumps, the relaxation of linear growth condition can be found in \cite{FangZ-05} and \cite{LanWu-14}.
 For   {\em   jump type} SDEs, can we relax   the usual linear growth condition as well? In this paper,  we provide   a sufficient condition in Theorem \ref{thm-no-explosion-sec-2} for non-explosion for solutions to  \eqref{eq-sde-sec-2} when the coefficients have super linear growth in a neighborhood of $\infty$.

Concerning the pathwise uniqueness, the usual argument is to use the (local) Lipschitz condition and Gronwall's inequality   to demonstrate that the $L_{2}$ distance $\E[|\wdt X(t) - X(t)|^{2}] $ between  two solutions $\wdt X, X$  vanishes if they have the same initial condition; see, for example, the proof of \citet[Theorem IV.9.1]{IkedaW-89}.
The   paper \cite{YamaW-71} relaxes the local Lipschitz condition to H\"older condition for one-dimensional SDEs without jumps. Since then, the problem of existence and pathwise uniqueness of solutions to SDEs with non-Lipschitz conditions has attracted growing attention. To name just a few, \cite{Bass-03} presents a sharp condition for existence and pathwise uniqueness for a one-dimensional SDE with a symmetric stable driving noise;   \cite{FuLi-10} and \cite{LiMyt-11} provide sufficient conditions for existence and pathwise uniqueness for one-dimensional  jump type SDEs with non-Lipschitz conditions; a crucial assumption in these two papers is that the kernel for the compensated Poisson integral term is  nondecreasing.  Such a nondecreasing kernel assumption was weakened in \cite{Four-13} and \cite{LiPu-12}. It should be noted that pathwise uniqueness need not hold in general if the diffusion matrix is merely uniformly nondegenerate,  bounded and continuous even in the one-dimensional case; see \cite{BassBC-04} for  such an example of one-dimensional SDE driven by a symmetric stable process in which pathwise uniqueness fails.   See also the discussion in \cite{TanakaTW-74}, in which pathwise uniqueness fails for some one-dimensional SDEs driven by   symmetric L\'evy   processes with H\"older continuous drift coefficients.  All the aforementioned references focus on one-dimensional SDEs and less is known for the multi-dimensional case. \cite{FangZ-05} establishes sufficient non-Lipschitz conditions for pathwise uniqueness  for multidimensional SDEs without jumps. These conditions were further relaxed in \cite{LanWu-14} using Euler's approximation method.
Further studies on jump type SDEs  with non-Lipschitz coefficients can be found in \cite{QiaoZ-08,Qiao-14,Priola-12,Priola-15}, among others.

This paper  aims to establish sufficient  non-Lipschitz conditions for pathwise uniqueness  for {\em multidimensional SDEs  with jumps}. Two sets of  sufficient non-Lipschitz  conditions (Assumptions \ref{assumption1-non-lip-sec-2} and \ref{assume-non-lip-sec-2}) for pathwise uniqueness are provided; both of them only require  the modulus of continuity  of the coefficients of  \eqref{eq-sde-sec-2} to hold  {\em locally} in a small neighborhood of the diagonal line $x=y$ on $\R^{d}\otimes\R^{d}$.
  As commented in  \cite{FangZ-05}, without Lipschitz condition, the usual argument for pathwise uniqueness  is not applicable.    When Assumption \ref{assumption1-non-lip-sec-2} holds, we  follow Yamada and Watanabe's  idea  and  construct a sequence of smooth functions to control the   $L_{1}$ distance 
  of two solutions $\wdt X, X$ up to an appropriately defined stopping time. Next we use a Bihari's inequality type argument to show that such an $L_{1}$ distance vanishes  if the two solutions start from the same initial conditions.
    Then we argue that
   $\wdt X(t)= X(t)$ a.s. for any $t \ge 0$,   which, in turn, leads to the desired pathwise uniqueness. The details are spelled out in Theorem \ref{thm1-path-uniqueness-sec-2}.
When Assumption \ref{assume-non-lip-sec-2} is in force, we develop a   quite different and more direct proof in Theorem \ref{thm-path-uniqueness-sec-2}.  In lieu of a sequence of smooth functions, a single smooth function is used to estimate, roughly speaking,   a ``scaled''  $L_{2}$ distance of two solutions to     \eqref{eq-sde-sec-2}, which helps us to immediately obtain  $\wdt X(t)= X(t)$ a.s. Example \ref{example1} is provided to demonstrate the utility of our results.

Now suppose  \eqref{eq-sde-sec-2} has a unique non-exploding  strong solution for any initial condition.
We say that the solution $X$ of \eqref{eq-sde-sec-2} satisfies the {\em non confluence} property, if for all $x\neq y\in \R^{d}$,
\begin{displaymath}
\P\{X^{x}(t) \neq X^{y}(t), \text{ for all } t \ge 0 \} =1,
\end{displaymath} where $X^{x}$ and $X^{y}$ denote solutions to \eqref{eq-sde-sec-2} with initial conditions $x$ and $y$, respectively.    We refer to \cite{FangZ-05} and \cite{LanWu-14} for sufficient conditions for non confluence for SDEs without jumps.  The recent paper \cite{Dong-18} contains some sufficient conditions for non confluence for jump SDEs. The key assumption in \cite{Dong-18} is on the jumps: for each $u \in U$, the function $  x\mapsto x+ c(x,u)$ is homeomorphic and that its inverse satisfies the linear growth and Lipschitz conditions. Such conditions are quite strong and not easy to verify in practice.  We aim to relax such conditions in this paper.  First,  as long as the function  $x\mapsto x+ c(x,u)$ is one-to-one for $\nu$-almost all $u\in U$, Theorem \ref{thm-non confluence} proposes a set of sufficient conditions  in terms of the existence of a certain Lyapunov function for non confluence for   \eqref{eq-sde-sec-2}. Then in Corollary \ref{cor-non confluence}, we prove that under a  slightly stronger condition on  the function  $x\mapsto x+ c(x,u)$, the non confluence property holds if the coefficients of \eqref{eq-sde-sec-2} is Lipschitz continuous. Remark \ref{rem-35 about jump condition} demonstrates that our condition is quite easy to verify in general.

This paper next considers Feller and strong Feller properties for solutions to \eqref{eq-sde-sec-2} under non-Lipschitz conditions. Suppose \eqref{eq-sde-sec-2} has a  solution $X$ which is unique in the sense of probability law. For $f\in \B_{b}(\R^{d})$ (the set of bounded and measurable functions), set \begin{equation}
\label{eq-semi-group}
P_{t} f(x) : = \E_{x} [f(X(t))] = \E[f(X^{x}(t))], \quad t \ge 0, x\in \R^{d}.
\end{equation} The family of operators $\{ P_{t}\}_{t \ge0}$ forms a semigroup of bounded linear operators on $ \B_{b}(\R^{d})$.  We are interested in the continuous properties of the semigroup. The semigroup or the corresponding process is said to be {\em Feller} if $P_{t} $ maps $C_{b}(\R^{d})$ (the set of bounded and continuous functions) into itself and {\em strong Feller} if it maps $ \B_{b}(\R^{d})$ into $ C_{b}(\R^{d})$ for each $t >0$.  Most work on Feller and strong Feller properties assumes (local) Lipschitz conditions on the coefficients of the underlying processes; see, for example, Theorem 6.3.4 of \cite{Stroock-V} for diffusion processes,   Proposition 2.1 of \cite{Wang-10} for jump diffusions and Theorems 4.5 and 5.6 of \cite{Xi-09} for regime-switching jump diffusions. By contrast, this paper establishes these properties under non-Lipschitz conditions. Proposition \ref{prop-Feller-old} and  Theorem \ref{thm-feller new}  deal with Feller property while Theorem  \ref{thm-strong-Feller} and Proposition \ref{prop-str-Feller-local version} establish  strong Feller property. In these results, we only require certain {\em local} modulus of continuity  of the coefficients of  \eqref{eq-sde-sec-2}   in a small neighborhood of the diagonal line.  These results   improve substantially over the related work in the literature, even for SDEs without jumps. See Remark \ref{rem-about str-Feller conditions} for more details.
Our main tool in establishing these two theorems is the coupling method, 
which has been extensively applied  in the literature to study various properties of many processes,  see, for example, \cite{ChenLi-89,Lindvall,PriolaW-06,Wang-10} and the references therein.  

Next we take up the issue of exponential ergodicity for the process $X$ of \eqref{eq-sde-sec-2}. Following  the same approach as those in \cite{Zhang-09,PriolaSXZ-12,Qiao-14}, we first show that the process $X$ of \eqref{eq-sde-sec-2} is irreducible under Assumptions \ref{assumption-non-linear-growth-sec-2} and \ref{assume-non-lip-sec-2}. The conditions for irreducibility in \cite{Qiao-14} are somewhat relaxed here; see Remark \ref{rem-irreducibilty} for more details.  The irreducibility and strong Feller property together then imply the uniqueness of an invariant measure for the process $X$. A Foster-Lyapunov type drift condition then leads to the existence of an invariant measure  as well as the exponential ergodicity. The details are spelled out in Theorem \ref{thm-exp-ergodicity}.

 As applications, we  consider
 SDEs driven by multidimensional L\'evy processes $\d X(t) =\psi(X(t-)) \d L(t), $ in which $\psi: \R^{d}\mapsto \R^{d\times d}$ is Borel measurable and non-Lipschitz, and $L$ is a   multidimensional L\'evy process, e.g., a symmetric stable process of order $\alpha$ with $\alpha\in (0,2)$. Under what conditions on $\psi$ so that   this SDE has a unique non-exploding strong solution? We aim to answer this question in Section \ref{sect-Levy sde}.
  For another application, we consider a  Cauchy problem related to a L\'evy type operator \eqref{eq-Levy type operator}. Our goal is to establish a non-standard  Feynman-Kac formula for solutions to the Cauchy problem and therefore establish a connection between integral-differential equations and SDEs of the form \eqref{eq-sde-sec-2}. The details are spelled out in Section  \ref{sect-Feynman-Kac}.

 Upon the completion of the manuscript, we learned that the  recent paper \cite{Dong-18} also contains sufficient conditions for non-explosion, pathwise uniqueness and non confluence for jump type SDEs. These conditions are quite different from our corresponding conditions and they do not seem to imply one another.  In addition, the methodologies in \cite{Dong-18} and this paper have different flavors, even though certain technical aspects are similar.

 The rest of the paper is organized as follows. Section \ref{sect-non-explosion-pathwise} presents sufficient conditions for non-explosion and pathwise uniqueness. The non confluence property for solutions to \eqref{eq-sde-sec-2} is investigated in Section \ref{sect-non confluence}.  Section \ref{sect-Feller} is focused on Feller property under non-Lipschitz condition. Strong Feller property is treated in Section \ref{sect-strong-Feller}. Section \ref{sect: irreducibility and ergodicity} studies irreducibility and exponential ergodicity.  Finally Section \ref{sect-applications}    studies   SDEs driven by multidimensional L\'evy processes and establishes a Feynman-Kac formula for L\'evy type operators. Several technical proofs are arranged in Appendix \ref{sect-appendix}.

To facilitate the
 presentation, we introduce some notation
 that will be used often in later sections.
   Throughout the paper, we use
 $\lan x, y\ran$  or $x\cdot y$ interchangeably to denote the inner product of
   the vectors $x$ and $y$ with compatible dimensions. If
   $A$ is a vector or matrix, let $A^{T}$ denote the transpose of $A$ and set $|A|:=\sqrt{\tr(AA^{T})}$.
     For a sufficiently smooth function $\phi: \R^d \to \R$, $D_{x_i} \phi= \frac{\partial \phi}{\partial x_i}$,
    $D_{x_ix_j} \phi= \frac{\partial^2 \phi}{\partial x_i\partial x_j}$, and we denote by $D\phi   =(D_{x_1}\phi, \dots, D_{x_d}\phi)^{T}\in \R^{d}$ and $D^2\phi =(D_{x_ix_j}\phi) \in \R^{d\times d}$ the gradient and  Hessian matrix of $\phi$, respectively.
    For $k \in \mathbb N$, $C^{k}(\R^{d})$ is the collection of functions $f: \R^{d}\mapsto \R$
    with continuous partial derivatives up to the $k$th order while  $C^{k}_{c} (\R^{d})$ denotes the space of $C^{k}$ functions with compact support.
 If $B $ is a set, we use   $I_B$ to denote the   indicator function of $B$. Throughout the paper, we adopt the conventions that $\sup \emptyset =-\infty$
 and $\inf \emptyset = + \infty$. Finally, we note that the infinitesimal generator $\LL$ of \eqref{eq-sde-sec-2} is given by \begin{equation}
\label{eq-L-operator-sect-2}\begin{aligned}
\LL f(x) : =& \  \lan Df(x), b(x) \ran + \frac12\tr\bigl(\sigma(x)\sigma(x)^{T}D^{2} f(x)\bigr) \\
    & + \int_{U} \bigl[ f(x+ c(x,u)) -f(x) - \lan Df(x), c(x,u) \ran \bigr]\nu(\d u), \ \ f\in C_{c}^{2}(\R^{d}). \end{aligned}
\end{equation}

\section{Nonexplosion and Pathwise Uniqueness}\label{sect-non-explosion-pathwise}

In this section, we consider nonexplosion and pathwise uniqueness for   SDE  \eqref{eq-sde-sec-2}.
Assume throughout this paper that  the functions
$b(\cdot)$,   $\sigma(\cdot)$, and $c(\cdot,u)$ (for each $u \in U$) are continuous
 and that $c(\cdot,\cdot)$  is Borel measurable such that   the function $x\mapsto \int_{U}|c(x,u)|^{2}\nu(\d u)$ is continuous. 
 For the convenience of later presentations, let us recall several important notions from \cite{IkedaW-89} (as well as the presentations in \cite{Situ-05}).
Let  $(\Omega, \F, \{\F_{t}\}_{t \ge 0}, \P)$ be a filtered probability space satisfying the usual hypotheses. Let $W  = \{W(t), t \ge 0\}$  be a standard $d$-dimensional $\{\F_{t}\}$-Brownian motion and let $p=\{p(t), t \ge 0\}$   be an  $\{\F_{t}\}$-Poisson point processes on $U$  with characteristic measures $\nu(\d u)$, where as mentioned in the introduction,   $(U, \mathfrak U)$  is a measurable space and $\nu$ a $\sigma$-finite measure on $U$. Suppose that  $W$ and $p$ are independent. Let $N (\d s, \d u)$   be the Poisson random measures associated with $p$ and let $\widetilde N(\d s, \d u)$ be the compensated Poisson random measure of $N (\d s, \d u)$.    By a {\em weak solution up to an explosion time} to \eqref{eq-sde-sec-2}, we mean an $\R^{d}$-valued  \cadlag and $\{\F_{t}\}$-adapted  process $X = \{X(t), t \ge 0\}$  such that the equation  \begin{align*}
X(t\wedge\tau_{n})  = X(0) &+ \int_{0}^{t\wedge \tau_{n}}\!b(X(s))  \d s + \int_{0}^{t\wedge \tau_{n}}\!\!\sigma(X(s))  \d W(s)      + \int_{0}^{t\wedge \tau_{n}} \!\!\!\! \int_{U} c(X(s-),u)  \wdt N(\d s, \d u)
\end{align*}  holds  for all $n \in \mathbb N$ and  $t \ge 0$ a.s., where the {\em initial condition} $X(0) \in \F_{0}$ and $\tau_{n}:= \inf\{t \ge 0: |X(t)| > n\} $ is the first exit time from the closed ball $B(n):=\{x\in\R^d: |x| \le n\}$.  Clearly the sequence $\{\tau_{n}, n \in \mathbb N\}$ is nondecreasing. The limit  $\tau: = \lim_{n\to \infty} \tau_{n}$, finite or infinite,  is called the  {\em explosion time} or  {\em lifetime}  for the process $X$.  In particular, we say that $X$ is   {\em explosive} if $\P\{ \tau  < \infty\} > 0$; otherwise, $X$ is said to be {\em non-explosive}.  We say {\em pathwise uniqueness}  holds for \eqref{eq-sde-sec-2} if for any two solutions $X_{1}, X_{2}$
 of the equation satisfying  $\P\{ X_{1}(0) = X_{2}(0)\} =1$ we have  $\P\{X_{1}(t) = X_{2}(t) \text{ for all  } t \ge 0\} =1$.  Let $\{\G_{t}\}_{t \ge 0}$ be the augmented natural filtration generated by $W$ and $p$. A solution $X$ of \eqref{eq-sde-sec-2} is called a {\em strong solution} if it is adapted with respect to $\{\G_{t}\}_{t \ge 0}$.

The classical results (e.g., \cite{IkedaW-89}) indicate that if the coefficients  satisfy the usual linear growth condition, then the solution to \eqref{eq-sde-sec-2}   is non-explosive. This section aims to relax the linear growth condition.
\begin{Assumption}\label{assumption-non-linear-growth-sec-2}
There exists a  nondecreasing  function $\zeta:[0,\infty) \mapsto [1,\infty)$  
 that is continuously differentiable and  satisfies
\begin{equation}\label{eq-zeta-sec-2}
 \int_{0}^{\infty} \frac{\d r}{ r\zeta(r) + 1} =\infty,
\end{equation}
such that for all 
$x\in \R^{d}$,
\begin{align}
\label{eq-coeffs-non-linear-growth-sec-2}
& 2 \lan x, b(x)\ran + |\sigma(x)|^{2} + \int_{U}  |c(x,u)|^{2} \nu(\d u)  \le \kappa [|x|^{2}\zeta(|x|^{2}) + 1],
\end{align} where $\kappa $ is a positive constant.
\end{Assumption}

Some common functions satisfying \eqref{eq-zeta-sec-2} include $\zeta(r) =1$,   $\zeta(r) = \log r$ and $  \zeta(r) = \log r \log(\log r) $ for $r$ large.
\begin{thm}\label{thm-no-explosion-sec-2}
Under Assumption \ref{assumption-non-linear-growth-sec-2}, any solution to \eqref{eq-sde-sec-2}  is non-explosive.
\end{thm}
\begin{proof} This proof is motivated by the proof of Theorem A in \cite{FangZ-05}.
Consider the function $\phi(r) : = \exp\{\int_{0}^{r} \frac{\d z}{z\zeta(z)  +1} \}$ for $r> 0$.  Then we have $$\phi'(r) =  \frac{\phi(r)}{r\zeta(r) +1}> 0, \text{ and } \phi''(r) =  \phi(r) \frac{1-\zeta(r) - r \zeta'(r)}{(r\zeta(r) + 1)^{2}}. $$ Since $\zeta(r) \ge 1$ and $\zeta$ is nondecreasing, it follows that $\phi''(r) \le 0$ and hence $\phi$ is a concave function.
On the other hand, thanks to 
\eqref{eq-zeta-sec-2}, we have $\phi(r) \to\infty$ as $r\to \infty.$

Now consider the function $\Phi: \R^{d}\mapsto \R^{+}$  defined by  $\Phi(x)=  \phi(|x|^{2})$. We have  $\Phi(x) \to \infty$ as $|x|\to \infty$. Moreover,   straightforward computations lead to
$D\Phi(x) = 2\phi'(|x|^{2}) x$  and  $D^{2} \Phi(x) = 2\phi'(|x|^{2})  I +  4\phi''(|x|^{2})  xx^{T}.$
Since $\phi$ is concave, we have $\phi(r) \le \phi(r_{0}) + \phi'(r_{0}) (r-r_{0})$ for all $r, r_{0} \in (0,\infty)$. Using this inequality with $r_{0}= |x|^{2}$ and $r = |x +  c(x,u)|^{2}$, we have  \begin{displaymath}
\phi(|x+ c(x,u)|^{2}) - \phi(|x|^{2})   \le \phi'(|x|^{2}) [|x+ c(x,u)|^{2} - |x|^{2}] = \phi'(|x|^{2})  [2 \lan x, c(x,u)\ran + |c(x,u)|^{2}].
\end{displaymath} Then it follows that
\begin{align*}
    \int_{U}& [\Phi(x+ c(x,u)) - \Phi(x) - \lan D\Phi(x), c(x,u)\ran ]\nu(\d u) \\
    & =    \int_{U}[ \phi(|x+ c(x,u)|^{2}) - \phi(|x|^{2})   - 2 \phi'(|x|^{2})  \lan x, c(x,u)\ran ]\nu (\d u) \\
 & \le    \int_{U} \bigl[ \phi'(|x|^{2})  [2 \lan x, c(x,u)\ran + |c(x,u)|^{2}] -2 \phi'(|x|^{2})  \lan x, c(x,u)\ran \bigr]\nu(\d u)\\
 & =  \int_{U} \phi'(|x|^{2})  |c(x,u)|^{2}\nu(\d u).
 \end{align*}
Consequently we can compute\begin{align*}
 \LL \Phi(x) & = 2 \phi'(|x|^{2})  \lan x, b(x)\ran +  \frac12 \tr\Bigl(\sigma(x)\sigma'(x)\bigl [2\phi'(|x|^{2})  I +  4\phi''(|x|^{2})  xx^{T}\bigr]\Bigr)\\
  & \qquad  + \int_{U}[\Phi(x+ c(x,u)) - \Phi(x) - \lan D\Phi(x), c(x,u)\ran ]\nu (\d u)\\
 & \le  \phi'(|x|^{2})\bigg(2 \lan x, b(x)\ran + |\sigma(x)|^{2} + \int_{U} |c(x,u)|^{2} \nu(\d u)\bigg) +2 \phi''(|x|) |\lan x, \sigma(x)\ran|^{2}      \\
& \le \frac{\phi(|x|^{2})}{|x|^{2} \zeta(|x|^{2}) + 1} \kappa (|x|^{2} \zeta(|x|^{2}) + 1)
  \le \kappa \phi(|x|^{2})  =  \kappa \Phi(x),
\end{align*} where we used \eqref{eq-coeffs-non-linear-growth-sec-2} and the fact that $\phi''(r) \le 0$ to derive the  second inequality.
  The rest of the proof is quite standard: one can apply It\^o's formula and the optional sampling theorem to the process $\{e^{-\kappa t} \Phi(X(t)), t \ge 0\}$ to argue that $\P\{ \lim_{n\to\infty} \tau_{n} =\infty\} =1$. Indeed similar arguments can be found in, e.g.,  the proofs of Theorem 2.1 of \cite{MeynT-93III}, Theorem A of \cite{FangZ-05}, and Theorem 2.1 of \cite{Dong-18}. We shall omit the details here.
\end{proof}



The rest of the section is focused on sufficient conditions for pathwise uniqueness for the stochastic differential equation \eqref{eq-sde-sec-2}. Let us first make the following assumption:
\begin{Assumption}\label{assumption1-non-lip-sec-2}
There exist a positive constant $\delta_{0}$ and a nondecreasing and concave function $\rho: [0,\infty)\mapsto [0,\infty)$ satisfying  $\rho(r) > 0$  for $ r > 0$,
  and  \begin{equation}
\label{eq-rho-sec-2}\int_{0+}\frac{ \d r}{ \rho(r)} = \infty,
\end{equation}
such that for all  $R > 0$ and $x,z\in \R^{d}$ with $|x| \vee |z| \le R$ and $|x-z| \le\delta_{0}$,
 \begin{align}
\label{eq-coeffs-non-lip-sec-2}
&2 \lan z-x, b(z)-b(x)\ran + |\sigma(z) -\sigma(x)|^{2} \le \kappa_{R} |z- x| \rho(|z-x|), \\
\label{eq-integral-term-non-lip-sec-2}& \int_{U} |c(z,u)-c(x,u)| \nu(\d u) \le \kappa_{R} \rho(|z-x|),
\end{align} where $\kappa_{R} $ is a positive constant. In addition, assume $\int_{U}|c(0,u)|\nu(\d u)<\infty$.
\end{Assumption}

\begin{thm}\label{thm1-path-uniqueness-sec-2}
Under Assumptions  \ref{assumption-non-linear-growth-sec-2} and
\ref{assumption1-non-lip-sec-2}, pathwise uniqueness holds for \eqref{eq-sde-sec-2}. 
\end{thm}

The proof of Theorem \ref{thm1-path-uniqueness-sec-2} is in the same spirit of  Yamada and Watanabe's argument for pathwise uniqueness  in  \cite{YamaW-71} and \cite{FuLi-10,LiMyt-11}.  The key idea is to construct a sequence of monotone $C^{2}$ functions $\{\psi_{n}\}$ satisfying certain conditions so that one can bound  the growth of the $L_{1}$ distance $\E[|\wdt X(t\wedge S_{\delta_{0}}) - X(t\wedge S_{\delta_{0}})|] $ of two solutions $\wdt X, X$ with the same initial condition, where $S_{\delta_{0}}$  is a stopping time related to the solutions $\wdt X, X$. Next we use a Bihari's inequality type argument to obtain   $\E[|\wdt X(t\wedge S_{\delta_{0}}) - X(t\wedge S_{\delta_{0}})|] =0$, from which we derive $\wdt X(t)= X(t)$ a.s. This, together with the right-continuity of solutions to  \eqref{eq-sde-sec-2}, enables us to establish the pathwise uniqueness result. To preserve the flow of presentation, we relegate the proof of Theorem \ref{thm1-path-uniqueness-sec-2} to  Appendix \ref{sect-appendix}.


Next we propose a   different assumption than that of Assumption \ref{assumption1-non-lip-sec-2}  for pathwise uniqueness.
\begin{Assumption}\label{assume-non-lip-sec-2}
There exist a positive number $\delta_{0}$  and  a nondecreasing   and  concave function $\varrho: [0,\infty)\mapsto [0,\infty)$ satisfying \begin{equation}
\label{eq-varrho-conditions-Feller}
0 < \varrho(r) \le  (1+r)^{2}\varrho(r/(1+r)) \text{ for all } r >  0,\quad  \text{ and }\quad \int_{0^{+}} \frac{\d r}{\varrho(r)} = \infty,
\end{equation}
such that for all  $R > 0$ and  $x,z\in \R^{d}$ with  $|x| \vee |z| \le R$ and $|x-z| \le \delta_{0}$, \begin{equation}
\label{eq-Feller-condition} \begin{aligned}
2   \lan x-z,  \ b(x)-b(z)\ran +  |\sigma(x) - \sigma(z)|^{2} 
 + \int_{U}| c(x,u)-c(z,u) |^{2}  \nu(\d u) \le  \kappa_{R} \varrho(|x-z|^{2}),\end{aligned}
\end{equation}  where $\kappa_{R}$ is a positive constant.
 \end{Assumption}

Some common functions  satisfying Assumptions \ref{assumption1-non-lip-sec-2} and \ref{assume-non-lip-sec-2} include $\varrho(r) = r$ and concave and increasing  functions such as $\varrho(r) = r \log(1/r)$,  $\varrho(r) = r \log( \log(1/r))$, and $\varrho(r) = r \log(1/r)  \log( \log(1/r))$ for $r \in (0,\delta)$ with $\delta > 0$ small enough.
It is worth pointing out that \eqref{eq-coeffs-non-lip-sec-2} and \eqref{eq-integral-term-non-lip-sec-2} in Assumption  \ref{assumption1-non-lip-sec-2}  and  \eqref{eq-Feller-condition} in Assumption \ref{assume-non-lip-sec-2}  only require the modulus continuity to hold in a small neighborhood of the diagonal line $x=z$ in $\R^{d}\otimes\R^{d}$ with $|x|\vee |z| \le R$ for each $R > 0$.  This is in contrast to  those in \cite{FuLi-10,LiMyt-11}. Note, in particular, that the constant $\kappa_{R}$ in \eqref{eq-coeffs-non-lip-sec-2},  \eqref{eq-integral-term-non-lip-sec-2} and \eqref{eq-Feller-condition} may depend on $R$. These conditions are very general but make our analysis very subtle; careful analysis are required to accommodate  various stopping times.
On the other hand,   even in the case with $\varrho (r) =r$,  since $\nu(U)$ is not necessarily  finite, Assumptions \ref{assumption1-non-lip-sec-2} and \ref{assume-non-lip-sec-2} in general cannot imply each other.  Moreover, instead of using  a sequence of $C^{2} $ functions $\{\psi_{n}\}$, we use a single $C^{2}$ function $H$  to obtain 
the desired pathwise uniqueness result in Theorem \ref{thm-path-uniqueness-sec-2}.  Compared with the aforementioned references, the proof of Theorem \ref{thm-path-uniqueness-sec-2} is simpler and more direct. Again, we arrange the proof of Theorem \ref{thm-path-uniqueness-sec-2} to Appendix \ref{sect-appendix}.

\begin{thm}\label{thm-path-uniqueness-sec-2}
Under Assumptions  \ref{assumption-non-linear-growth-sec-2} and
\ref{assume-non-lip-sec-2}, pathwise uniqueness holds for \eqref{eq-sde-sec-2}. 
\end{thm}

\begin{rem}
In case  that the solution to \eqref{eq-sde-sec-2} has a finite explosion time with positive probability, then pathwise uniqueness holds up to the explosion time under Assumptions \ref{assumption1-non-lip-sec-2} or \ref{assume-non-lip-sec-2}.
\end{rem}

\begin{thm}\label{thm-existence+uniquness}
Suppose  Assumption \ref{assumption-non-linear-growth-sec-2} and either Assumption \ref{assumption1-non-lip-sec-2} or Assumption \ref{assume-non-lip-sec-2} hold.
Then for any $x\in \R^{d}$, \eqref{eq-sde-sec-2} has a unique strong non-explosive solution $X=\{X(t), t\ge 0\}$ satisfying $X(0) =x$.
\end{thm}
\begin{proof} Suppose that Assumptions \ref{assumption-non-linear-growth-sec-2} and  \ref{assumption1-non-lip-sec-2}   hold; the proof for the case under Assumptions \ref{assumption-non-linear-growth-sec-2} and \ref{assume-non-lip-sec-2}   is similar.   Let us fix some $x\in \R^{d}$. 
For each $n\in \N$ with $|x| < n$, let $\psi_{n}: \R^{d}\to [0,1]$ be a $C^{\infty}$ function such that $\psi_{n}(x) = 1$ for $|x| \le n$ and $\psi_{n}(x) =0$ for $|x| \ge n+1$.
Define $b_{n} := \psi_{n} b, \sigma_{n}: = \psi_{n} \sigma $ and $c_{n} : =\psi_{n} c$. Then
\begin{equation}
\label{eq-bsgc-mn}
b_{m}(x)=b_{n}(x) = b(x), \  \sigma_{m}(x) = \sigma_{n}(x) = \sigma(x), \text{ and }c_{m}(x,u)=c_{n}(x,u) = c(x,u)
\end{equation}    for all $\{ (x,u) \in \R^{d}\times U:  |x| \le n \}  $  and $m \ge n$.
Obviously,  for each  $n\in \N$, $b_{n}(\cdot)$ and   $\sigma_{n}(\cdot)$ are bounded and continuous and that $c_{n}(\cdot,\cdot)$ is measurable.  Moreover,  for any $x\in \R^{d}$,   $M(x,B): = \nu\{u \in U: c_{n}(x,u) \in B\}, B \in \B(\R^{d})$,   is a $\sigma$-finite measure on $\B(\R^{d})$ and satisfies \begin{align*} \int_{\R^{d}} \frac{|y|^{2}}{1+ |y|^{2}} M(x,\d y) &  =   \int_{U} \frac{|c_{n}(x,u)|^{2}}{1+ |c_{n}(x,u)|^{2}} \nu(\d u)   \le \int_{U} \psi_{n}(x)^{2}|c (x,u)|^{2} \nu(\d u) < \infty.      \end{align*}  Likewise,  for any $\phi \in C_{b}(\R^{d})$, the function \begin{align*}\int_{\R^{d}} \frac{|y|^{2}}{1+ |y|^{2}} \phi(y) M(x,\d y) &= \int_{U} \frac{|c_{n}(x,u)|^{2}}{1+ |c_{n}(x,u)|^{2}} \phi(c_{n}(x,u))\nu(\d u)\\ &  \le \|\phi\|_{\infty} \int_{U} \psi_{n}(x)^{2}|c (x,u)|^{2} \nu(\d u)
\le \kappa  \psi_{n}(x)^{2} \|\phi\|_{\infty}  [|x|^{2} \zeta(|x|^{2}) + 1]  \end{align*}
is bounded and continuous,  where the last inequality follows from
\eqref{eq-coeffs-non-linear-growth-sec-2} in
Assumption \ref{assumption-non-linear-growth-sec-2}.
Now consider the the operator \begin{displaymath} \begin{aligned}
 \LL_{n}f(x) & :  =    \lan Df(x), b_{n}(x) \ran + \frac12\tr\bigl(\sigma_{n}(x)\sigma_{n}(x)^{T}D^{2} f(x)\bigr) \\
    & \qquad + \int_{U} \bigl[ f(x+ c_{n}(x,u)) -f(x) - \lan Df(x), c_{n}(x,u) \ran \bigr]\nu(\d u) \\
     & = \lan Df(x), b_{n}(x) \ran + \frac12\tr\bigl(\sigma_{n}(x)\sigma_{n}(x)^{T}D^{2} f(x)\bigr) \\
     & \qquad  + \int_{\R^{d}} [f(x+ y) - f(x) - \lan Df(x), y\ran] M(x,\d y), \quad f\in C_{c}^{2}(\R^{d}).  \end{aligned}
\end{displaymath}
Thanks to Theorem 2.2 in \cite{Stroock-75},  the martingale problem for $\LL_{n}$  has a solution. Then by virtue of Theorem 2.3 of \cite{Kurtz-11}, the stochastic differential equation
\begin{equation}
\label{eq-sde n-sec 2}
X^{(n)}(t) = x + \int_{0}^{t} b_{n}(X^{(n)}(s)) \d s + \int_{0}^{t} \sigma_{n}(X^{(n)}(s)) \d W(s) + \int_{0}^{t}\int_{U} c_{n}(X^{(n)}(s-), u) \wdt N(\d s,\d u)\end{equation} has a weak solution $X^{(n)}$.

Apparently $b_{n}$ and $\sigma_{n}$ satisfy Assumption \ref{assumption1-non-lip-sec-2}. On the other hand, for all $x,z\in \R^{d}$ with $|x|\vee |z| \le R$ and $|x-z| \le \delta_{0}$, we have from \eqref{eq-integral-term-non-lip-sec-2} that \begin{align}\label{eq-cn difference estimate}
\nonumber  \int_{U}\!   |c_{n}(x,u)-c_{n}(z,u)|\nu(\d u) &\le    \int_{U} \! \big[  |\psi_{n}(x)| |c(x,u)-c(z,u)| + |\psi_{n}(x) - \psi_{n}(z)| |c(z,u)|\big]\nu(\d u)   \\
\nonumber    &   \le   \int_{U}  \!  |c(x,u)-c(z,u)|\nu(\d u) +  |\psi_{n}(x) - \psi_{n}(z)|\int_{U} |c(z,u)|\nu(\d u)\\
    & \le \kappa_{R} \rho(|x-z|)  + K_{R} |x-z|,
  \end{align} where 
 we used the facts that $\psi_{n}$ is locally Lipschitz and that the function $x\mapsto \int_{U} |c(x,u)|\nu(\d u)$ is locally bounded to obtain the last  inequality. Furthermore, since $\rho(\cdot)$ is concave and $\rho(0) =0$, it follows that $\rho(r) \ge \frac{\rho(\delta_{0})}{\delta_{0}} r$ or $r \le \frac{\delta_{0}} {\rho(\delta_{0})} \rho(r)$ for all $r \in [0,\delta_{0}]$. Applying this observation in \eqref{eq-cn difference estimate} leads to   $$\int_{U}   |c_{n}(x,u)-c_{n}(z,u)|\nu(\d u) \le  \wdt\kappa_{R} \rho(|x-z|),  $$   for all $x,z\in \R^{d}$ with $|x|\vee |z| \le R$ and $|x-z| \le \delta_{0}$,
  where $\wdt \kappa_{R}$  is a  positive constant. Therefore   $c_{n}$ also satisfies Assumption \ref{assumption1-non-lip-sec-2}.
Theorem  \ref{thm1-path-uniqueness-sec-2} then   implies that pathwise uniqueness holds.
   Now by Theorem 2 of \cite{BarczyLP-15},  for each  $n\in \N$, a unique strong solution $X^{(n)}$ to \eqref{eq-sde n-sec 2} exists. Let $\tau_{n}:=\inf\{ t \ge 0: |X^{(n)}(t)| > n\}$ denote the  first exit time  of  $X^{(n)}$  from  $B(n) $.

 Furthermore, for any   $m\ge n$,    again thanks to the pathwise uniqueness as well as \eqref{eq-bsgc-mn},  the   processes $X^{(m)}$ and $X^{(n)}$   have the same first exit time $\tau_{n}$ from  $B(n) $
  and   $X^{(m)} (t) = X^{(n)}(t)$ for all $t < \tau_{n}$.  Now the process $X$ defined by $X(t) : = X^{(n)}(t)$ for all $t < \tau_{n}$,  $n \in \N$ is the unique strong solution to \eqref{eq-sde-sec-2} with $X(0) =x$;  Theorem \ref{thm-no-explosion-sec-2} implies that $X$ has no finite explosion time. This completes the proof.
\end{proof}

\begin{cor}\label{cor-existence+uniqueness general sde} Let   $U_{0}\subset U$ so that   $\nu(U\setminus U_{0}) < \infty$.  Suppose Assumption \ref{assumption-non-linear-growth-sec-2} and either Assumption \ref{assumption1-non-lip-sec-2} or Assumption \ref{assume-non-lip-sec-2}  (with $U$   replaced by $U_{0}$) hold. Then  for any initial condition $x\in \R^{d}$, the stochastic differential equation
\begin{equation}
\label{eq-sde-2}
\begin{aligned}
X(t) & = x+ \int_{0}^{t} b(X(s)) \d s +\int_{0}^{t} \sigma(X(s))\d W(s) \\ & \qquad+ \int_{0}^{t}\int_{U_{0}} c(X(s-),u)\wdt N(\d s, \d u) +\int_{0}^{t}  \int_{U\setminus U_{0}} c(X(s-),u) N(\d s, \d u)
\end{aligned}
\end{equation}
     has a unique strong non-explosive solution $X=\{X(t), t\ge 0\}$.
\end{cor} \begin{proof}
This corollary follows from the standard interlacing procedure as in the proof Theorem 6.2.9 of \cite{APPLEBAUM}.  Indeed, under Assumptions \ref{assumption-non-linear-growth-sec-2} and   \ref{assumption1-non-lip-sec-2} or Assumptions  \ref{assumption-non-linear-growth-sec-2} and \ref{assume-non-lip-sec-2}  (with $U$   replaced by $U_{0}$), for any initial condition, Theorem  \ref{thm-existence+uniquness} implies that the SDE \begin{displaymath}
\d Y(t) = b(Y(t)) \d t + \sigma(Y(t))\d W(t) + \int_{U_{0}} c(Y(t-),u)\wdt N(\d t, \d u)
\end{displaymath} has a unique strong non-exploding solution. Next we use the interlacing procedure as in the proof Theorem 6.2.9 of \cite{APPLEBAUM} to construct a solution to \eqref{eq-sde-2}. The solution is unique thanks to Theorems \ref{thm1-path-uniqueness-sec-2} or \ref{thm-path-uniqueness-sec-2} and the interlacing structure.
\end{proof}

\begin{exm}\label{example1}
Let us consider the following SDE \begin{equation}
\label{eq-example1}
\d X(t) = b(X(t))\d t + \sigma(X(t))\d W(t) + \int_{U} c(X(t-), u) \wdt N(\d t, \d u), X(0) = x \in \R^{3},
\end{equation} where $W$ is a 3-dimensional standard Brownian motion, $\wdt N(\d t,\d u)$ is a compensated Poisson random measure with compensator $\d t\,\nu(\d u)$ on $[0,\infty)\times U$, in which $U=  \{u\in \R^{3}: 0 < |u| <1 \}$ and $\nu(\d u) : = \frac{\d u}{|u|^{3+\alpha}}$ for some $\alpha\in (0,2)$. The coefficients of \eqref{eq-example1} are given by
\begin{displaymath}
b(x)=\begin{pmatrix}-x_{1}^{1/3}-x_{1}^{3}\\ -x_{2}^{1/3}-x_{2}^{3}\\-x_{3}^{1/3}-x_{3}^{3}\\ \end{pmatrix}\!, \
\sigma(x) = \begin{pmatrix}\frac{x_{1}^{2/3}}{\sqrt 2} + 1 & \frac{x_{2}^{2}}{3} & \frac{x_{3}^{2}}{3}\\    \frac{x_{1}^{2}}{3} &\frac{x_{2}^{2/3}}{\sqrt 2} + 1& \frac{x_{3}^{2}}{3} \\ \frac{x_{1}^{2}}{3} & \frac{x_{2}^{2}}{3} & \frac{x_{3}^{2/3}}{\sqrt 2} + 1& \end{pmatrix}\!,\
c(x,u) = \begin{pmatrix} \gamma x_{1}^{2/3} |u| \\  \gamma x_{2}^{2/3} |u|\\ \gamma x_{3}^{2/3} |u| \end{pmatrix}\!,
\end{displaymath} in which $\gamma$ is a positive constant so that $\gamma^{2} \int_{U} |u|^{2} \nu(\d u) = \frac12$.

Note that even without jumps, the coefficients of \eqref{eq-example1} do not satisfy  conditions (H1) and
(H2) in  \cite{FangZ-05} since   $\sigma$ and $b$ grow very fast in the neighborhood of $\infty$ and they
are H\"older continuous with orders $\frac23$
and $ \frac13$, respectively.
Nevertheless,   the coefficients of \eqref{eq-example1} still satisfy  Assumptions \ref{assumption-non-linear-growth-sec-2}  and  \ref{assume-non-lip-sec-2} and hence a unique non-exploding strong solution of \eqref{eq-example1} exists.  The verifications of these assumptions are as follows.
\begin{align}\label{eq1-ex1-non-lin-growth}
 \nonumber2& \lan x, b(x)\ran +|\sigma(x)|^{2} + \int_{U} |c(x,u)|^{2} \nu(\d u)        \\
  \nonumber  &    = 2 \sum_{j=1}^{3} x_{j}\bigl(-x_{j}^{1/3}-x_{j}^{3}\bigr)  + \sum_{j=1}^{3} \biggl(\frac12 x_{j}^{4/3} + \frac{2}{9} x_{j}^{4} +\sqrt 2 x_{j}^{2/3}+ 1\biggr) + \int_{U} \gamma^{2} |u|^{2} \sum_{j=1}^{3} x_{j}^{4/3} \nu(\d u)\\
    & = - \frac{16}{9}  \sum_{j=1}^{3} x_{j}^{4} - \sum_{j=1}^{3} x_{j}^{4/3} + \sqrt 2 \sum_{j=1}^{3} x_{j}^{2/3}  + 3. 
\end{align} This verifies Assumption \ref{assumption-non-linear-growth-sec-2}.  For the verification of Assumption \ref{assume-non-lip-sec-2}, we compute
\begin{align}\label{eq2-ex1-non-Lip}
 \nonumber 2& \lan x-y, b(x)-b(y)\ran +|\sigma(x)-\sigma(y)|^{2} + \int_{U} |c(x,u)-c(y,u)|^{2} \nu(\d u)          \\
 \nonumber   & = -2  \sum_{j=1}^{3} (x_{j}- y_{j})(x_{j}^{1/3} - y_{j}^{1/3} + x_{j}^{3} - y_{j}^{3}) +\frac12  \sum_{j=1}^{3}(x_{j}^{2/3} - y_{j}^{2/3})^{2} \\  \nonumber & \qquad + \frac29\sum_{j=1}^{3}(x_{j}^{2}- y_{j}^{2})^{2} + \int_{U} \sum_{j=1}^{3} \gamma^{2} (x_{j}^{2/3} - y_{j}^{2/3})^{2} |u|^{2} \nu(\d u)\\
    & = - \frac{16}{9}\sum_{j=1}^{3}  (x_{j}- y_{j})^{2}\biggl[\biggl(x_{j}+\frac{7}{16}y_{j}\biggr)^{2} + \frac{207}{256} y_{j}^{2} \biggr]- \sum_{j=1}^{3}\bigl(x_{j}^{1/3}- y_{j}^{1/3}\bigr)^{2} \bigl(x_{j}^{2/3}+y_{j}^{2/3}\bigr).
    \end{align} Obviously  this verifies Assumption \ref{assume-non-lip-sec-2}.
\end{exm}

\section{Non Confluence Property}\label{sect-non confluence}
\begin{thm}\label{thm-non confluence}
Assume the conditions of Theorem \ref{thm-existence+uniquness}. In addition, suppose
\begin{equation}
\label{eq-jump neq 0 condition}\begin{aligned}
\text{for }\nu\text{-almost all }u, \text{ the function } x\mapsto x+ c(x,u) \text{ is one-to-one}.
\end{aligned}\end{equation}  Moreover, assume that there exist  
a nondecreasing and concave function $\psi: [0,\infty) \to [0,\infty)$ that vanishes only at $r=0$, and a   $C^{2} $ function $V:(0,\infty)\mapsto (0,\infty)$ satisfying
\begin{enumerate}
  \item[(i)]  $V$ is nonincreasing  in a neighborhood of $0$ and $\lim_{r\downarrow 0} V(r) = \infty$,  and
  \item[(ii)]  for all $ |x-z| >0 $, 
  \begin{align}
\label{eq-wdt-LL-V-non confluence}
\nonumber \psi(V(|x-z|)) \ge&\  \frac12  \biggl(V''(|x-z|) - \frac{V'(|x-z|)}{|x-z|} \biggr) \frac{|\lan x-z, \sigma(x)- \sigma(z)\ran|^{2}}{|x-z|^{2}}
 \\& +  \frac{V'(|x-z|)}{2|x-z|} \bigl(2 \lan x -z, b(x)-b(z)\ran + |\sigma(x) -\sigma(z)|^{2} \bigr)   \\
  \nonumber   &+  \int_{U} \bigg[ V(|x-z+c(x,u)- c(z,u)|) - V(|x-z|) \\
  \nonumber   & \qquad \qquad - \frac{V'(|x-z|)}{|x-z|}\lan x-z, c(x,u)-c(z,u)\ran\biggr]\nu(\d u).
\end{align}
\end{enumerate} 
Then the non confluence property for  \eqref{eq-sde-sec-2}  holds:
\begin{equation}
\label{eq-non confluence}
\text{If }\wdt x\neq x, \text{ then }\P\{X^{\wdt x}(t) \neq X^{x}(t) \text{ for all } t\ge0 \} =1,
\end{equation}   where $X^{\wdt x}$ and $X^{x}$ denote the solutions to \eqref{eq-sde-sec-2} with initial conditions $\wdt x$ and $ x$, respectively.
\end{thm}

\begin{rem}
Note that \eqref{eq-jump neq 0 condition} prevents  the process $  X^{\wdt x}(t)  - X^{x}(t)$ from  jumping  to $0$ from a nonzero location.
 Also, by   It\^o's formula, the right hand side of \eqref{eq-wdt-LL-V-non confluence} is the extended generator $\wdt \LL$ of the process $X^{\wdt x} - X^{x} $ applied to the function $(x-z)\mapsto V(|x-z|)$; see \cite{MeynT-93III} for the definition of the extended generator. We can also regard $\wdt \LL$ as the {\em basic coupling operator} of $\LL$ of \eqref{eq-L-operator-sect-2}; see Section \ref{sect-Feller} for more details.
\end{rem}

\begin{proof}[Proof of Theorem~\ref{thm-non confluence}]
Let $\wdt X(t)= X^{\wdt x}(t)$, $X(t) =X^{x}(t)$ and denote $\Delta_{t}: = \wdt X(t) - X(t)$    as in the proof of Theorem \ref{thm1-path-uniqueness-sec-2}. In addition, assume that  $|\Delta_{0}| = |\wdt x - x|> 0$. For all $\N\ni n >  \frac{1}{|\Delta_{0}|} $ and $R > |\wdt x | \vee |x| $, define
\begin{align*}
 & T_{{1}/{n}}  : = \inf\bigl\{ t\ge 0: |\Delta_{t}| \le {1}/{n}\bigr\},
 \text{ and } \tau_{R}  : =  \inf\{ t\ge 0: |\wdt X(t)| \vee |X(t) |> R\}.  \end{align*}
 Put $ T_{0} : =\inf\{ t\ge 0: |\Delta_{t}| =0\}$. Then we have $ T_{0} = \lim_{n\to \infty} T_{{1}/{n}}$ and $\lim_{R \to \infty}\tau_{R} = \infty$ a.s.
Applying It\^o's formula to the process
$V(|\Delta_{\cdot \wedge \tau_{R} \wedge T_{{1}/{n}}  }|)$
and using \eqref{eq-wdt-LL-V-non confluence}, we have
\begin{align*}
\E  [ V (|\Delta_{t\wedge \tau_{R} \wedge T_{{1}/{n}}  } |)]
   & = V(|\Delta _{0} |) + \E\biggl[\int_{0}^{t\wedge \tau_{R}\wedge T_{{1}/{n}}  } \wdt \LL V(|\wdt X(s) - X(s)|) \d s\biggr]\\
  & \le V(|\Delta _{0} |) + \E\biggl[\int_{0}^{t\wedge \tau_{R}\wedge T_{{1}/{n}}  }  \psi ( V(|\Delta_{s}|)) \d s\biggr]    \\
   &\le V(|\Delta _{0} |)  +  \E\biggl[\int_{0}^{t}  \psi (V(|\Delta_{s\wedge \tau_{R}\wedge T_{{1}/{n}}  }|)) \d s\biggr] \\
  & \le  V(|\Delta _{0} |)  +\int_{0}^{t}  \psi \big(\E[ V(|\Delta_{s\wedge \tau_{R}\wedge T_{{1}/{n}}  }|)]\big) \d s,
\end{align*} 
where we used the concavity of $\psi$ and Jensen's inequality to obtain the last inequality.
Denote $u(t) : =\E  [ V (|\Delta_{t\wedge \tau_{R} \wedge T_{{1}/{n}}  } |)] $. Then $u $ satisfies
$0 \le u(t) \le  V(|\Delta _{0} |)  + \int_{0}^{t} \psi(u(s))\d s. $
We can use a similar  argument as that  in the end of the proof of Theorem \ref{thm1-path-uniqueness-sec-2} to show that
\begin{equation}\label{eq-u(t,R)-con confluence}
0 \le u(t) =\E  [ V (|\Delta_{t\wedge \tau_{R} \wedge T_{{1}/{n}}  } |)]  \le G^{-1}(G(V(|\Delta _{0} |) ) + t),
\end{equation} in which $G(r) : =  \int_{1}^{r} \frac{\d s}{\psi(s)} $, $r\in [0,\infty)$ and $G^{-1}(y): =  \inf\{s \ge 0: G(s) > y \}$ for $y\in \R$.
   Note that since $\psi$ is nonnegative, both $G$ and $G^{-1}$ are nondecreasing. In addition, since $\infty > V(|\Delta _{0} |) > 0$, we have $\infty > G(V(|\Delta _{0} |) ) + t > -\infty$ and hence $G^{-1}(G(V(|\Delta _{0} |) ) + t)\ge 0$ is finite. Now letting $R\to \infty$ in \eqref{eq-u(t,R)-con confluence}, we obtain from    Fatou's lemma that   $$0 \le \E  [ V (|\Delta_{t  \wedge T_{{1}/{n}}  } |)]  \le G^{-1}(G(V(|\Delta _{0} |) ) + t).$$ Furthermore, on the set $\{T_{{1}/{n}} < t  \}$, $|\Delta_{t  \wedge T_{{1}/{n}}  } | \le {1}/{n}$. Thus it follows from  condition (i)   that \begin{align*}
   V\big({1}/{n}\big) \P\{T_{{1}/{n}} < t  \}& \le   \E  \big[ V (|\Delta_{t  \wedge T_{{1}/{n}}  } |)I_{\{T_{{1}/{n}} < t  \}}\big]
       \le   \E  [ V (|\Delta_{t  \wedge T_{{1}/{n}}  } |)]  \le G^{-1}(G(V(|\Delta _{0} |) ) + t).
    \end{align*} Rewrite the above inequality as \begin{displaymath}
\P\{T_{{1}/{n}} < t   \} \le \frac{G^{-1}(G(V(|\Delta _{0} |) ) + t)}{ V ({1}/{n} )}
\end{displaymath} Now passing to the limit as $n \to \infty$, we obtain from condition (i) that $\P\{T_{0} < t  \} =0$.  This is true for any $t \ge 0$ so letting $t \to \infty $, we obtain $\P\{T_{0} <  \infty\} =0$.  In other words, $|\Delta_{t}|$ is positive on the interval $[0, \infty) $  a.s.
This completes the proof.
\end{proof}

Theorem  \ref{thm-non confluence} presents sufficient condition for non confluence in terms of the existence of a certain Lyapunov function. Often,  it is not an easy task to find such a Lyapunov function. The following corollary indicates that as long as the coefficients of \eqref{eq-sde-sec-2} is Lipschitz, then the non confluence property holds.
\begin{cor}\label{cor-non confluence} Suppose  Assumption \ref{assumption-non-linear-growth-sec-2} 
and that there exists a $\delta> 0$ such that
\begin{equation}
\label{eq2-jump neq 0 condition}\begin{aligned}
\nu\big\{u\in U: &\text{ there exist } x,z\in \R^{d} \text{ such that } x- z \neq 0\\   &\qquad\text{ but } |x-z + c(x,u)-c(z,u)|   \le \delta|x-z|\big\} =0.
\end{aligned}\end{equation}
Assume the coefficients of \eqref{eq-sde-sec-2} satisfy for some positive constant $K$ that
\begin{equation}
\label{eq-Lip for non confluence}\begin{aligned}
2& |\lan x-z, b(x)-b(z)\ran| + |\sigma(x) - \sigma(z)|^{2}   \\ & + \int_{U} \big[  |c(x,u)- c(z,u)|^{2}+ | (x-z) \cdot (c(x,u)- c(z,u)) |\big] \nu(\d u) \le K |x-z|^{2},
\end{aligned}\end{equation} for all $x,z\in \R^{d}$. 
Then the non confluence property for  \eqref{eq-sde-sec-2}  holds.
\end{cor}

\begin{proof} Apparently \eqref{eq-Lip for non confluence} verifies  Assumption \ref{assume-non-lip-sec-2}. This, together with Assumption \ref{assumption-non-linear-growth-sec-2},  implies that  \eqref{eq-sde-sec-2} has a unique strong non-exploding solution $X^{x}$ for any initial condition $x\in \R^{d}$. Note also that \eqref{eq2-jump neq 0 condition} implies \eqref{eq-jump neq 0 condition}. The remaining proof is to find a smooth function $V$ satisfying the conditions of Theorem \ref{thm-non confluence}.

Consider the function $V(r): = r^{-2}$ for $ r > 0$. Of course $V$ satisfies condition (i) of Theorem \ref{thm-non confluence}. It remains to verify condition (ii).
 To this end, let us first prove that for all $x,y\in \R^{n}$ with $x\neq 0$ and  $|x+y| \ge \delta |x|, $ where $\delta> 0$ is some constant,  we have
\begin{equation}
\label{eq-|x|-2-estimate}
V(|x+y|) - V(|x|) - DV(|x|)\cdot y =   \frac{1}{|x+y|^{2}} - \frac{1}{|x|^{2}} + \frac{2  x\cdot y}{|x|^{4}} \le K \frac{ |y|^{2}\vee | x \cdot y |}{|x|^{4}},
\end{equation} in which $K$ is a positive constant.
Let us    prove \eqref{eq-|x|-2-estimate} in three cases:

{\bf Case 1: $x\cdot y \ge 0$}. In this case, it is easy to verify that for any $\theta\in [0,1]$, we have $|x+ \theta y |^{2} = |x|^{2} + 2 \theta x\cdot y + \theta^{2}|y|^{2} \ge |x|^{2}$. Therefore we can use  the Taylor expansion with integral reminder to compute
\begin{align*}
     \frac{1}{|x+y|^{2}} - \frac{1}{|x|^{2}} + \frac{2  x\cdot y}{|x|^{4}} &  = \int_{0}^{1} \frac12 y\cdot D^{2} V(x+ \theta y) y \, \d \theta \\
    &   =   \int_{0}^{1} \biggl[- \frac{|y|^{2}}{|x+ \theta y|^{4}}+ 2\frac{y^{T}(x+ \theta y)(x+ \theta y)^{T} y}{|x+ \theta y |^{6}} \biggr]\d \theta\\
     & \le 2\int_{0}^{1} \frac{|y|^{2}}{|x+ \theta y|^{4}} \d \theta \le 2\int_{0}^{1} \frac{|y|^{2}}{|x|^{4}} \d \theta =  \frac{2 |y|^{2}}{|x|^{4}}.
\end{align*}

{\bf Case 2: $x\cdot y  <  0$ and $2x\cdot y + |y|^{2} \ge 0$}.  In this case, we have $|x+y|^{2} = |x|^{2} + 2x\cdot y+ |y|^{2} \ge |x|^{2}$ and hence $|x+y|^{-2} - |x|^{-2} \le 0$; which together with $x\cdot y \le 0$ implies that $|x+y|^{-2} - |x|^{-2} + 2 |x|^{-4} x\cdot y \le 0.$

{\bf Case 3: $x\cdot y  <  0$ and $2x\cdot y + |y|^{2} <  0$}.  In this case, we use the bound   $|x+y|^{2}  \ge \delta^{2}|x|^{2}$  to compute
\begin{align*}
   \frac{1}{|x+y|^{2}} - \frac{1}{|x|^{2}} + \frac{2  x\cdot y}{|x|^{4}} &  =   \frac{|x|^{2} - |x+y|^{2}}{|x|^{2} |x+y|^{2}}  + \frac{2 x\cdot y}{|x|^{4}}
   \hspace{-9pt}& & = -\frac{|y|^{2}}{|x|^{2} |x+y|^{2}} -\frac{2 x\cdot y}{|x|^{2} |x+y|^{2}} +  \frac{2 x\cdot y}{|x|^{4}} \\
    & \le  -\frac{2 x\cdot y}{|x|^{2} |x+y|^{2}} && \le  -\frac{2 x\cdot y}{\delta^{2}|x|^{4}}.
\end{align*}

  Combining the three cases gives \eqref{eq-|x|-2-estimate}.

For all $x\neq z\in \R^{d}$, \eqref{eq2-jump neq 0 condition} implies that $\nu\{u \in U: |x-z+ c(x,u)-c(z,u)| \le \delta|x-z| \}=0 $.   Hence, with  the notations $\lbar A(x,z), B(x,z)$ defined in \eqref{eq-A-bar-notation},  we can use \eqref{eq-Lip for non confluence}  and  \eqref{eq-|x|-2-estimate} to compute
  \begin{align*}
 \wdt \LL V(|x-z|)  & =    \frac12    V''(|x-z|)   \lbar A(x,z)  
   +  \frac{V'(|x-z|)}{2|x-z|} \bigl(2 B(x,z)- \lbar A(x,z) 
 + |\sigma(x) -\sigma(z)|^{2} \bigr)   \\
 & \qquad + \int_{U}      \biggl[V(|x-z+ c(x,u)-c(z,u)|) - V(|x-z|) \\ & \qquad -\frac{V'(|x-z|)}{|x-z|} \lan x-z, c(x,u)-c(z,u)\ran \biggr] \nu (\d u)  \\
 & \le  \frac{4}{|x-z|^{4}} \cdot \frac{|x-z|^{2}| \sigma(x)- \sigma(z)|^{2}}{|x-z|^{2}} +  \frac{1}{|x-z|^{4}}2 \big|\lan x -z, b(x)-b(z)\ran\big|  
 \\
 & \qquad +K \int_{U}  \frac{| c(x,u)-c(z,u)|^{2}\vee \big| \lan x-z,  c(x,u)-c(z,u) \ran\big|}{|x-z|^{4}}\nu (\d u) \\
 & \le  K |x-z|^{-2},
\end{align*} where $K$ is some positive constant. This verifies condition (ii) of Theorem \ref{thm-non confluence} and hence finishes the proof of the corollary.
\end{proof}
\begin{rem}\label{rem-35 about jump condition}
Assume that either  $$
     \lan x-z, c(x,u)-c(z,u)\ran \ge 0, $$ or $$ \lan x-z, c(x,u)-c(z,u)\ran <0 \text{ and }2  \lan x-z, c(x,u)-c(z,u)\ran + |c(x,u)-c(z,u)|^{2} \ge 0,$$    for all $x,z\in\R^{d}$   and $u \in U$. Then \eqref{eq2-jump neq 0 condition} is automatically satisfied and moreover,  the integrand of the integral term in \eqref{eq-Lip for non confluence} can be replaced by $|c(x,u)-c(z,u)|^{2} $. This is clear from Cases 1 and 2 for the proof of \eqref{eq-|x|-2-estimate}. \end{rem}

\section{Feller Property}\label{sect-Feller}

\begin{Assumption}\label{Assumption-martingale-well-posed}
For any initial condition $x\in \R^{d}$, the stochastic differential equation \eqref{eq-sde-sec-2} has a non-exploding weak solution $X^{x}$ and the solution is unique in the sense of probability law.
\end{Assumption}

Under Assumption \ref{Assumption-martingale-well-posed}, we can define the semigroup $P_{t} f(x):= \E_{x}[f(X(t))] = \E[f(X^{x}(t))]$ for $f\in \B_{b}(\R^{d})$ and $t \ge 0$, where $X^{x} $ denotes the unique weak solution of \eqref{eq-sde-sec-2} with initial condition $X^{x}(0) =x \in\R^{d}$.

We have the following result:

\begin{prop}\label{prop-Feller-old}
Suppose that Assumption \ref{Assumption-martingale-well-posed} and either Assumption \ref{assumption1-non-lip-sec-2} or Assumption \ref{assume-non-lip-sec-2} hold, then the process $X$ is Feller continuous.
\end{prop}

\begin{proof} 
 Let Assumptions \ref{Assumption-martingale-well-posed} and  \ref{assumption1-non-lip-sec-2} hold and use the same notations as in the proof of Theorem \ref{thm1-path-uniqueness-sec-2}.  The end of
  the proof of Theorem \ref{thm1-path-uniqueness-sec-2} (cf. \eqref{eq-u(t) to 0}) reveals  that for any $R > 0$
 \begin{equation}
\label{eq-Delta_t->0 in mean}
 \lim_{|\wdt x-x| \to 0} \E[|\Delta_{t\wedge S_{\delta_{0}}\wedge \tau_{R}}| ] =\lim_{|\wdt x-x| \to 0} \E[|\wdt X(t\wedge S_{\delta_{0}}\wedge \tau_{R}) - X(t\wedge S_{\delta_{0}}\wedge \tau_{R})| ]  =0 \ \ \text{ for all } t \ge 0.
\end{equation} On the set $\{S_{\delta_{0}} \le t\wedge \tau_{R}\}$, we have $|\Delta_{t\wedge S_{\delta_{0}}\wedge \tau_{R}}| \ge \delta_{0}$ and hence
$ \delta_{0} \P\{S_{\delta_{0}} \le t\wedge \tau_{R} \} \le \E[|\Delta_{t\wedge S_{\delta_{0}}\wedge \tau_{R}}| ].$    
 For any $\epsilon > 0$ and  $t \ge 0$, we can choose an $R > 0$ sufficiently large so that $\P(\tau_{R} <  t) < \epsilon$.
For any $\e > 0$,
 we can compute\begin{align*}
  \P&\{|\Delta_{t} | > \e \}  \\& =  \P\{|\Delta_{t} | > \e, \tau_{R} <  t \} +   \P\{|\Delta_{t} | > \e, \tau_{R} \ge  t, S_{\delta_{0}} >  t \} +  \P\{|\Delta_{t} | > \e, \tau_{R} \ge t, S_{\delta_{0}} \le t \}\\
   & \le \epsilon+  \P\{|\Delta_{t\wedge S_{\delta_{0}}\wedge \tau_{R}} | > \e, \tau_{R} \ge  t, S_{\delta_{0}} >  t \} + \P \{ S_{\delta_{0}} \le t \wedge \tau_{R}\}\\
   & \le \epsilon + \P\{|\Delta_{t\wedge S_{\delta_{0}}\wedge \tau_{R}} | > \e\}  + \frac{ \E[|\Delta_{t\wedge S_{\delta_{0}}\wedge \tau_{R}}| ] }{\delta_{0}}\\
   & \le \epsilon + \frac{\E[|\Delta_{t\wedge S_{\delta_{0}}\wedge \tau_{R}}| ]}{\e} + \frac{\E[|\Delta_{t\wedge S_{\delta_{0}}\wedge \tau_{R}}| ]}{\delta_{0}} \\
   &  \to \epsilon + 0,
\end{align*}    as $\wdt x -x \to 0$, where we used \eqref{eq-Delta_t->0 in mean} in the last step. Since $\epsilon > 0$ is arbitrary,   it follows from  that $\Delta_{t}$ converges to 0 in probability as $\wdt x -x \to 0$.

Recall that $\Delta_{t} = \wdt X(t) - X(t)$, in which $\wdt X$ and $X$ denote the solutions to \eqref{eq-sde-sec-2} with initial conditions $\wdt x$ and $x$, respectively. Thus we see that $\wdt X(t)$ converges to $X(t)$ in probability as $\wdt x \to x $.  For any $f \in C_{b}(\R^{d})$, the  mapping theorem (see, e.g., Theorem 2.7 of \cite{Billinsley-99}) implies that $f(\wdt X(t))$ converges weakly to $f(X(t))$ as $\wdt x \to x $. The bounded convergence theorem further implies that $\E[f(\wdt X(t))] \to \E[f(X(t))]$ as $\wdt x \to x$. The    Feller continuity therefore follows.

 Similar argument leads to the Feller continuity under Assumptions  \ref{Assumption-martingale-well-posed} and \ref{assume-non-lip-sec-2} as well.
\end{proof}

 Assumptions \ref{assumption1-non-lip-sec-2} and \ref{assume-non-lip-sec-2}   impose continuity conditions   on $\int_{U} |c(x,u)- c(z,u)|\nu(\d u)$ and  $\int_{U} |c(x,u)-c(z,u)|^{2}\nu(\d u)$, respectively.  These conditions are sometimes restrictive for the function $c$ and the L\'evy measure $\nu$.  For example, suppose $U= \R_{0}^{d}$, $\nu(\d u) = \frac{\d u}{|u|^{d+ \alpha}}$, in which $\alpha\in (1,2)$, and $c(x,u) =c(x) u$ with $c(x)\in \R^{d\times d}$ being a non-constant matrix. In such a case, we have $c(x,u)- c(z,u)= (c(x)-c(z))u$ and thus both $\int_{U} |c(x,u)- c(z,u)|\nu(\d u)$ and  $\int_{U} |c(x,u)- c(z,u)|^{2}\nu(\d u)$ may diverge to $\infty$. Then neither Assumptions \ref{assumption1-non-lip-sec-2} nor \ref{assume-non-lip-sec-2} can be applied to derive the Feller continuity.  We wish to relax such conditions and thus improve Proposition \ref{prop-Feller-old}.

\begin{Assumption}\label{assumption-feller-new condition} There exist a positive constant $\delta_{0}  $ and a nondecreasing and concave function $\varrho : [0,\infty) \mapsto [0,\infty)$ satisfying \eqref{eq-varrho-conditions-Feller}
such that
\begin{equation}
\label{eq-feller-new condition}\begin{aligned}
\int_{U}&  \bigl[ |c(x,u)-c(z,u)|^{2} \wedge (4 |x-z| \cdot |c(x,u) -c(z,u)|)\bigr] \nu(\d u) \\  &   \pushright{ +\,  2 \lan x-z, b(x)-b(z)\ran + |\sg(x)-\sg(z)|^{2 }     \le 2 \kappa_{R} |x-z| \varrho (|x-z|)}
\end{aligned}\end{equation} for all $x,z\in \R^{d}$  with  $ |x|\vee |z|  \le R$ and  $|x-z| \le \delta_{0}$, where $\kappa_{R}$ is  a positive constant.
\end{Assumption}

Apparently Assumption \ref{assumption-feller-new condition} relaxes the conditions on $c$ and $\nu$ over those in Assumptions \ref{assumption1-non-lip-sec-2} and \ref{assume-non-lip-sec-2}.
The main result of this section is:

\begin{thm}\label{thm-feller new}
Suppose Assumptions \ref{Assumption-martingale-well-posed} and \ref{assumption-feller-new condition} hold. Then the process $X$ is Feller continuous.
\end{thm}

We will use    the coupling method to prove Theorem \ref{thm-feller new}.
To this end, we recall the  infinitesimal generator $\LL$ of  \eqref{eq-sde-sec-2} defined in \eqref{eq-L-operator-sect-2}. 
 To construct the basic coupling operator for $\LL$,
  let us first introduce some notations.
  For $x,z\in \R^{d}  $, we set
$$a(x,z)=\begin{pmatrix}
a(x) & \sigma (x) \sigma (z)^{T} \\
\sigma (z) \sigma (x)^{T} & a(z)
\end{pmatrix},\quad
b(x,z)=\begin{pmatrix}
b(x)\\
b(z) \end{pmatrix},$$ where $a(x) = \sigma (x)\sigma (x)^{T}$
 and $ a(z)$ is similarly defined.
 Next we define the basic coupling operator (\cite{Chen04,Wang-10}) for the operator $\LL$ of \eqref{eq-L-operator-sect-2}
 \begin{equation}
\label{eq-A-coupling-operator} \begin{aligned}
\wdt{ \mathcal{L}} & f(x,z)  : =\! \bigl[ \wdt \Omega_{\text{diffusion}}   + \wdt \Omega_{\text{jump}}    \bigr] f(x,z),
\end{aligned}\end{equation} where  $f(x, z) \in C_{c}^{2} (\R^{d}  \times \R^{d}  )$,  and
\begin{align}\label{eq-Omega-d-defn}
&\wdt {\Omega}_{\text{diffusion}}f(x,z) =\frac
{1}{2}\hbox{tr}\bigl(a(x,z)D^{2}f(x, z)\bigr)+\langle
b(x,z), D f(x,z)\rangle,
\\
\label{eq-Omega-j-defn}
&\begin{aligned}
     \displaystyle\wdt {\Omega}_{\text{jump}}  f(x,z)
 &   = \int_{U}
[f(x+c(x,      u),       z+c(z,u))-f(x,      z)  \\
  & \qquad\quad -  \langle D_{x } f(x,      z),   c(x,      u)\rangle - \langle D_{z} f(x,       z),   c(z,u)\rangle  ] \nu(\d u). \end{aligned}
\end{align}
Here and below,  $Df(x,z)$ represents the gradient of $f$ with respect to the variables $x$ and $z$, that is,  $Df(x,z) = (D_{x}f(x,z), D_{z}f(x,z))'$. Likewise,  $D^{2}f(x,z)$ denotes the Hessian of $f$ with respect to   $x$ and $z$.

\begin{lem}\label{lem-wdt-LL-F estimate}
Suppose Assumption \ref{assumption-feller-new condition} holds. Then  \begin{equation}\label{eq-wdt-LL-F estimate}
\wdt \LL F(|x-z|) \le   \kappa_{R} \varrho (F(|x-z|))
\end{equation} for all $x,z\in \R^{d}$ with  $|x| \vee |z| \le R$ and $0 < |x-z| \le \delta_{0}$, where the function $F$ is defined by $F(r): =\frac{r}{1+ r}, r \ge 0$.
\end{lem}
The proof of Lemma \ref{lem-wdt-LL-F estimate} involves straightforward but lengthy computations. To preserve the flow of the presentation, we arrange it in Appendix \ref{sect-appendix}.

\begin{proof}[Proof of Theorem \ref{thm-feller new}] By virtue of Theorem 5.6 in \cite{Chen04}, it suffices to prove that \begin{equation}
\label{eq:Wasserstein-convergence}
W_{d}(P(t, x,\cdot), P(t,z,\cdot) ) \to 0 \text{ as } z\to x,
\end{equation}
where  $\{P(t,x,\cdot): t > 0, x\in \R^{d}\}$ is the transition probability family associated with the process $X$ of \eqref{eq-sde-sec-2} and $W_{d}(\cdot, \cdot)$ denotes the Wasserstein metric between two probability measures: $$W_{d} (\mu,\nu): = \inf\biggl \{\int d(x,y) \pi(\d x, \d y): \pi \in \mathscr{C}(\mu, \nu) \biggr\},$$ where  $ \mathscr{C}(\mu, \nu)$ denotes the family of  coupling measures of  $\mu$ and $\nu$, and $d(x,y):= \frac{|x-y|}{1+|x-y|}$ for $x,y\in \R^{d}$.

Given $x\neq z$ with $\delta_{0} > |x-z | > \frac{1}{n_{0}}$, where $n_{0}\in \mathbb N$,   let $(\wdt X,\wdt Z)$ be the coupling    process corresponding to the operator $\wdt \LL$ of \eqref{eq-A-coupling-operator} with $(\wdt X(0),\wdt Z(0))=(x,z)$. Denote by $T$ the coupling time. For $n\ge n_{0}$ and $  R >  |x| \vee |z|$, 
define
\begin{align}\label{Tn tauR defns}
 & T_n : =   \inf \Bigl\{  t \ge 0:   | \wdt X  (t) - \wdt Z  (t) | < \frac{1}{n}\Bigr\}, \ \
    &  \tau_{R} : =  \inf \{  t \ge 0:   | \wdt X (t) | \vee |\wdt  Z  (t) | > R\}, \end{align}
    and \begin{align}\label{S delta defn}
     S_{\delta_{0}}: = \inf\{t \ge 0:  |\wdt  X  (t) - \wdt  Z  (t) | > \delta_{0} \}.
\end{align} We have $\tau_{R}\to \infty$ and $T_{n}\to T$ a.s. as $R\to \infty$ and $n \to \infty$, respectively. Moreover, by It\^o's formula and \eqref{eq-wdt-LL-F estimate}, we have
\begin{align*}
\E&[F(|\wdt X(t\wedge T_{n} \wedge S_{\delta_{0}} \wedge \tau_{R})- \wdt Z(t\wedge T_{n} \wedge S_{\delta_{0}} \wedge \tau_{R})|)] \\
   & = F(|x-z|) + \E\biggl[\int_{0}^{t\wedge T_{n} \wedge S_{\delta_{0}} \wedge \tau_{R}} \wdt \LL F(|\wdt X(s) - \wdt Z(s)|) \d s \biggr]\\
   & \le  F(|x-z|) + \kappa_{R}  \E\biggl[\int_{0}^{t\wedge T_{n} \wedge S_{\delta_{0}} \wedge \tau_{R}} \varrho  (F(|\wdt X(s) - \wdt Z(s)|))\d s \biggr].
   \end{align*}
   Now passing to the limit as $n \to \infty$, it follows from the bounded and monotone convergence theorems that
  \begin{align*}
\E& [F(|\wdt X(t\wedge T \wedge S_{\delta_{0}}\wedge \tau_{R} )- \wdt Z(t \wedge T \wedge S_{\delta_{0}}\wedge \tau_{R} )|)] \\
& \le  F(|x-z|) + \kappa_{R}  \E\biggl[\int_{0}^{t\wedge T \wedge S_{\delta_{0}} \wedge \tau_{R}} \varrho  (F(|\wdt X(s) - \wdt Z(s)|))\d s \biggr]\\
   & \le F(|x-z|) + \kappa_{R}  \E\biggl[\int_{0}^{t} \varrho  (F(|\wdt X(s\wedge T \wedge S_{\delta_{0}} \wedge \tau_{R}) - \wdt Z(s\wedge T \wedge S_{\delta_{0}} \wedge \tau_{R})|))\d s \biggr]\\
   &  \le  F(|x-z|) + \kappa_{R}\int_{0}^{t} \varrho \bigl(\E[F(|\wdt X(s\wedge T \wedge S_{\delta_{0}}\wedge \tau_{R}) - \wdt Z(s\wedge T \wedge S_{\delta_{0}}\wedge \tau_{R})|)]\bigr)\d s,
   \end{align*}
 where we use the concavity of $\varrho $ and Jensen's inequality to obtain the last inequality.
 Then    using 
 Bihari's inequality, we have
\begin{displaymath}
\E [F(|\wdt X(t\wedge T \wedge S_{\delta_{0}}\wedge \tau_{R} )- \wdt Z(t \wedge T \wedge S_{\delta_{0}} \wedge \tau_{R})|)]  \le G^{-1} (G\circ F(|x-z|) + \kappa_{R} t),
\end{displaymath}  where the function $G(r): = \int_{1}^{r} \frac{\d s} {\varrho (s)}$ is strictly increasing and satisfies $G(r) \to -\infty$ as $r \downarrow 0$.  In addition, since the function $F$ is strictly increasing, we have
\begin{align*}
   F(\delta_{0}) \P\{S_{\delta_{0}} < t \wedge T\wedge \tau_{R} \} & \le  \E [F(|\wdt X(t\wedge T \wedge S_{\delta_{0}} \wedge \tau_{R})- \wdt Z(t \wedge T \wedge S_{\delta_{0}} \wedge \tau_{R})|) I_{\{ S_{\delta_{0}} < t \wedge T\wedge \tau_{R}\}}]    \\
    &   \le \E [F(|\wdt X(t\wedge T \wedge S_{\delta_{0}} \wedge \tau_{R})- \wdt Z(t \wedge T \wedge S_{\delta_{0}} \wedge \tau_{R})|)]  \\
    & \le G^{-1} (G\circ F(|x-z|) + \kappa_{R} t).
\end{align*}

 For any $t \ge 0$ and $\e > 0$, since $\lim_{R\to\infty} \tau_{R} = \infty$ a.s., we can choose some $R > 0$ sufficiently large so that $\P(t > \tau_{R}) < \e$.
Then it follows that
\begin{align*}
  \E& [F(|\wdt X(t) - \wdt Z(t)|)] \\ & =   \E [F(|\wdt X(t\wedge \tau_{R}) - \wdt Z(t\wedge \tau_{R})|)I_{\{t \le \tau_{R} \}}] +  \E [F(|\wdt X(t) - \wdt Z(t)|) I_{\{ t > \tau_{R}\}}]  \\       & \le   \E[F(|\wdt X(t\wedge T\wedge \tau_{R}) - \wdt Z(t\wedge T\wedge \tau_{R})|)]   + \e   \\
    &   = \E[F(|\wdt X(t\wedge T\wedge \tau_{R}) - \wdt Z(t\wedge T\wedge \tau_{R})|)I_{\{ S_{\delta_{0}} < t \wedge T\wedge \tau_{R}\}}]  \\ & \quad + \E[F(|\wdt X(t\wedge T\wedge \tau_{R}) - \wdt Z(t\wedge T\wedge \tau_{R})|)I_{\{ S_{\delta_{0}} \ge  t \wedge T\wedge \tau_{R}\}}]  + \e\\
    & \le \P\{S_{\delta_{0}} < t \wedge T\wedge \tau_{R} \} + \E [F(|\wdt X(t\wedge T\wedge \tau_{R}\wedge S_{\delta_{0}} )- \wdt Z(t \wedge T\wedge \tau_{R} \wedge S_{\delta_{0}} )|)] + \e\\
    & \le \frac{1 +2 \delta_{0}}{\delta_{0}} G^{-1} (G\circ F(|x-z|) + \kappa_{R} t) + \e.
\end{align*} Now passing to the limit, we obtain $\lim_{x-z \to 0} \E [F(|\wdt X(t) - \wdt Z(t)|)] \le 0+ \e$. 
  Since $\e >0$ is arbitrary, it follows that  $\lim_{x-z \to 0}\E[F(|\wdt X(t) - \wdt Z(t)| ) ]  = 0$.
   This leads to \eqref{eq:Wasserstein-convergence} because by the definition of $W_{d}$, we have
$W_{d}(P(t,x,\cdot), P(t,z,\cdot) )  \le  \E[F(|\wdt X(t ) - \wdt Z (t )|)].$
 This gives the Feller property as desired.
\end{proof}

\section{Strong Feller Property}\label{sect-strong-Feller}

\begin{Assumption}\label{Assump-Uniform-Elliptic}
There exists a $\lambda_0 > 0$ such that $\langle \xi, a(x)\xi\rangle \ge \lambda_0|\xi|^2$ for all $x,\xi\in \R^d$, where $a(x) : = \sigma(x) \sigma(x)^{T}$.  Denote by $\sigma_{\lambda_{0}}$ the unique symmetric nonnegative definite matrix-valued function such that $\sgla^{2} = a- \lambda_{0}I$.   In addition, there exist   positive constants  $\delta_{0}, \kappa_{0}  $ and a nonnegative function $\vartheta $ defined on $[0,\delta_{0}]$ satisfying $\lim_{r\to 0} \vartheta (r) =0$ such that
\begin{equation}
\label{eq-str-feller-condition}\begin{aligned}
\int_{U}&  \bigl[ |c(x,u)-c(z,u)|^{2} \wedge (4 |x-z| \cdot |c(x,u) -c(z,u)|)\bigr] \nu(\d u) \\  & \quad + 2  \lan x-z, b(x)-b(z) \ran + |\sgla(x)-\sgla(z)|^{2 }     \le 2 \kappa_{0} |x-z| \vartheta(|x-z|),
\end{aligned}\end{equation} for all $x,z\in \R^{d}$ with $|x-z| \le \delta_{0}$.
\end{Assumption}

The main result of this section is:
\begin{thm}\label{thm-strong-Feller}
Under Assumptions \ref{Assumption-martingale-well-posed} and  \ref{Assump-Uniform-Elliptic}, for any $t > 0$ and $f\in \B_{b}(\R^{d})$, we have
\begin{displaymath}
\sup_{x\neq z} \frac{|P_{t} f(x) - P_{t} f(z)|}{|x-z|} \le K \| f\|_{\infty},
\end{displaymath} where $K= K(t,\delta_{0}, \kappa_{0})$ is a positive constant. In particular, it follows that  the process $X$ of \eqref{eq-sde-sec-2} is strong Feller continuous.
\end{thm}

As in \cite{Wang-10,ChenLi-89,PriolaW-06}, we construct the coupling by reflection operator $\wdh \LL$ of $\LL$ as follows. For $x,z \in \R^{d}$, put $g(x,z ) : = -\lambda_{0} I + \sigma_{\lambda_{0}}(x)\sigma_{\lambda_{0}}(z)^{T}$
and  set
$$\wdh a(x,z )= \begin{pmatrix}
a(x ) & g(x,z ) \\
g(x,z )^{T}  & a(z )
\end{pmatrix},\quad
b(x,z )=\begin{pmatrix}
b(x )\\
b(z ) \end{pmatrix}.$$  We can verify directly that $\wdh a(x,z)$ is symmetric and nonnegative
definite. 
Then we define \begin{displaymath}
\wdh \Omega_{\text{diffusion}}  h(x,z) := \frac{1}{2} \tr(\wdh a(x,z) D^{2}  h(x,z)) + \langle b(x,z), D h(x,z)\rangle,
\end{displaymath} and \begin{equation}\label{eq-L-hat-defn}
\wdh \LL h(x,z) : = \wdh  \Omega_{\text{diffusion}}  h(x,z) + \wdt \Omega_{\text{jump}}  h(x,z),
\end{equation} where $h \in C_{0}^{2} (\R^{d}\times \R^{d})$ and $\wdt \Omega_{\text{jump}}  $ is defined in \eqref{eq-Omega-j-defn}. 
Let \begin{align*}
 & A(x,z) =  a(x) + a(z) - 2 g(x,z),\ \
    & \lbar A_{\lambda_{0}}(x,z) = \frac{1}{|x-z|^{2}} \langle x-z, A(x,z) (x-z)\rangle.
\end{align*} Then straightforward computations lead to   \begin{align*}
    \tr( A(x,z) ) =|\sgla(x)-\sgla(z)|^{2 } + 4 \lambda_{0}  \text{ and }
    \lbar A_{\lambda_{0}}(x,z) \ge 4 \lambda_{0} . \end{align*}

We need the following lemma to prove Theorem \ref{thm-strong-Feller}:
    \begin{lem}\label{lem-str-Feller}
Under Assumption \ref{Assump-Uniform-Elliptic}, there exist some positive constants $\beta$ and $\delta$ such that
\begin{equation}
\label{eq-L-hat-bound}
\wdh \LL F(|x-z|) \le -  \beta < 0
\end{equation} for all $x,z\in \R^{d}$ with $0 < |x-z| \le \delta$, where the function $F$ is defined by   $F(r): = \frac{r}{1+r}$, $r \ge 0$.
\end{lem}
\begin{proof}
We have  $F'(r) = \frac{1}{(1+r)^{2}} > 0$ and $F''(r) = \frac{-2}{(1+r)^{3}}  <  0$ for all $r \ge 0$.
Moreover we can verify directly that  for all $x,z \in \R^{d}$ with $0<|x-z| \le \delta_{0}$,
\begin{align}\label{eq1-strFe}
\nonumber \wdh \Omega_{\text{diffusion}}   F(|x-z|) & = \frac{F'' (|x-z|)}{2}\lbar A_{\lambda_{0}} (x,z) + \frac{F'(|x-z|) }{2|x-z|} \big[ \tr(A(x,z))-\lbar A_{\lambda_{0}}(x,z) + 2  B(x,z)\big]\\
   \nonumber  & \le 2 \lambda_{0}F'' (|x-z|) +\frac{F'(|x-z|) }{2|x-z|} [|\sgla(x)-\sgla(z)|^{2 } + 2B(x,z)]\\
     & \le \frac{-4\lambda_{0}}{(1+|x-z|)^{3}} +\frac{ \kappa_{0}}{(1+|x-z|)^{2}} \vartheta (|x-z|),
\end{align} where the last inequality follows from \eqref{eq-str-feller-condition}.

Then it follows from \eqref{eq-Omega jump F estimate} and \eqref{eq-str-feller-condition} that  for all $x,z \in \R^{d}$ with $0<|x-z| \le \delta_{0}$, we have  \begin{equation}
\label{eq2-str-fe-computation}
\begin{aligned}
\wdt {\Omega}_{\text{jump}} F(|x-z|)
& = \int_{U} \biggl[F(|x+c(x,u)-z-c(z,u)|) - F(|x-z|)\\
    &  \qquad \qquad
   - \frac{F'(|x-z|)}{|x-z|}  \lan x-z, c(x,u) - c(z,u)\ran  \biggr] \nu(\d u) \\
   & \le \frac{ \kappa_{0}}{(1+|x-z|)^{2}}\vartheta (|x-z|).\end{aligned}
\end{equation}
Plugging \eqref{eq1-strFe} and \eqref{eq2-str-fe-computation} into \eqref{eq-L-hat-defn} leads to
\begin{equation}
\label{eq-str-fe-estimation}
\begin{aligned}
\wdh \LL F(|x-z|) & \le \frac{-4\lambda_{0}}{(1+|x-z|)^{3}} +\frac{2 \kappa_{0}}{(1+|x-z|)^{2}} \vartheta (|x-z|) \\
  & \le \frac{-4\lambda_{0}}{(1+\delta_{0})^{3}}
  + 2\kappa_{0}\vartheta (|x-z|) 
  \le \frac{-2\lambda_{0}}{(1+\delta_{0})^{3}} < 0,
\end{aligned}\end{equation} for all  $x,z \in \R^{d}$ with $0<|x-z| \le \delta$, where $0 < \delta \le \delta_{0}$, whose existence follows from the assumption that $\lim_{r\to 0} \vartheta (r) =0$. This establishes \eqref{eq-L-hat-bound} and hence completes the proof. \end{proof}

\begin{proof}[Proof of Theorem \ref{thm-strong-Feller}] Let $\beta, \delta$ and $F$ be as in  Lemma \ref{lem-str-Feller}.
Given $x\neq z$ with $\delta > |x-z | > \frac{1}{n_{0}}$,
where $n_{0}\in \mathbb N$, let $(\wdt X,\wdt Z)$ be the coupling process
corresponding to the operator $\wdh \LL$ of \eqref{eq-L-hat-defn}
with $(\wdt X(0),\wdt Z(0))=(x,z)$.
Denote by $T$ the coupling time. For $\mathbb N \ni n\ge n_{0}$ and $R > 0$, 
define the stopping times $T_{n} $ and $\tau_{R}$ as in \eqref{Tn tauR defns}. Also define $S_{\delta}$ as in \eqref{S delta defn} (with $\delta_{0}$ replaced by $\delta$).
    We have
\begin{align*}
0 & \le F (\delta)\P \set{T_{n}\wedge \tau_{R} > S_{\delta}}
 \le  \E  \bigl[F ( | \wdt X  (T_{n}\wedge S_{\delta} \wedge \tau_{R} ) - \wdt Z  (T_{n}\wedge S_{\delta} \wedge \tau_{R}) |)\bigr] \\
  & = F(|x-z|) + \E\biggl[\int_{0}^{T_{n}\wedge S_{\delta} \wedge \tau_{R}} \wdh\LL F(|\wdt X(s)- \wdt Z(s)|) \d s \biggr]
  \le F (|x-z|) -\beta\E[T_{n}\wedge S_{\delta} \wedge \tau_{R}].
\end{align*} Then it follows that
\begin{equation}\label{eq-prob mean bdd by F}
F (\delta)\P  \set{T_{n}\wedge \tau_{R} > S_{\delta}} + \beta  \E [T_{n}\wedge S_{\delta} \wedge \tau_{R}] \le F (|x-z|).
\end{equation}Since $T_{n} \to T$ a.s. as $n \to \infty$ and $\tau_{R} \to \infty$ a.s. as $R \to \infty$,  we have
\begin{displaymath}
F (\delta)\P  \set{T  > S_{\delta}} +  \beta \E [T \wedge S_{\delta} ] \le F (|x-z|).
\end{displaymath}
Then for any $t >0$ and $0 < |x-z| < \delta$,
\begin{align*}
 \P  \set{T > t}    & = \P \set{T > t, S_{\delta} > t} + \P \set{T > t, S_{\delta} \le t}
       \le  \P \set{T  \wedge S_{\delta} > t} + \P \set{T >  S_{\delta} }  \\
    & \le \frac{1}{t} \E [T  \wedge S_{\delta} ] +  \P \set{T >  S_{\delta} }
     \le \biggl( \frac{1}{t\beta} + \frac{1}{F (\delta)}\biggr)  F  (|x-z|).
\end{align*} Finally, for any $f \in \B_{b}(\R^{d})$,   $t>0$, and $0 < |x-z| < \delta$, we  can write
\begin{align*}
 | P_{t} f(x) - P_{t} f(z)|  &  = |\E[ f(\wdt X(t))  - f(\wdt Z(t))]|
       \le 2\|f\|_{\infty} \P\{T> t\} \\
         & \le 2\|f\|_{\infty}  \biggl( \frac{1}{t\beta} + \frac{1}{F (\delta)}\biggr)  F  (|x-z|)=2\|f\|_{\infty}  \biggl( \frac{1}{t\beta} +  \frac{1+\delta}{\delta}\biggr)\frac{|x-z|}{1+|x-z|}\\
         & \le 2\|f\|_{\infty}  \biggl( \frac{1}{t\beta} + \frac{1+\delta}{\delta}\biggr)|x-z|.
\end{align*}
On the other hand, if $ |x-z| \ge \delta$, then we can write
\begin{displaymath}
 | P_{t} f(x) - P_{t} f(z)|  \le 2\|f\|_{\infty} \le 2\|f\|_{\infty}\frac{|x-z|}{\delta}.
\end{displaymath} We can combine the above two displayed equations to obtain
\begin{displaymath}
\frac{ | P_{t} f(x) - P_{t} f(z)|}{|x-z|} \le 2\|f\|_{\infty} \biggl[ \biggl( \frac{1}{t\beta} + \frac{1+\delta}{\delta}\biggr) \vee \frac{1}{\delta}\biggr] = 2\|f\|_{\infty} \biggl( \frac{1}{t\beta} + \frac{1+\delta}{\delta}\biggr).
\end{displaymath} In particular, the desired strong Feller property follows.
\end{proof}

 In view of Theorem \ref{thm-feller new}, one may naturally ask  whether the strong Feller property holds under a ``localized'' version of  Assumption \ref{Assump-Uniform-Elliptic}? The following result gives an affirmative answer:
  \begin{prop}\label{prop-str-Feller-local version} Let Assumption \ref{Assumption-martingale-well-posed} hold.
Suppose that for each $R > 0$, there exist positive constants $\lambda_{R}$ and $\kappa_{R}$ such that for all $x, z\in \R^{d}$ with $|x| \vee |z| \le R$, we have \begin{align*} & \lan \xi, a(x)\xi\ran \ge \lambda_{R}|\xi |^{2}, \quad   \forall  \xi \in \R,\end{align*} and \begin{align*}
\int_{U}&  \bigl[ |c(x,u)-c(z,u)|^{2} \wedge (4 |x-z| \cdot |c(x,u) -c(z,u)|)\bigr] \nu(\d u) \\  &  + 2  \lan x-z, b(x)-b(z) \ran + |\sg_{\lambda_{R}}(x)-\sg_{\la_{R}}(z)|^{2 }     \le 2 \kappa_{R} |x-z| \vartheta(|x-z|), \quad \forall |x-z| \le \delta_{0},
 \end{align*}  where $\delta_{0}$ is a positive constant and $\vartheta$ is a function satisfying the conditions specified in Assumption \ref{Assump-Uniform-Elliptic}, and $\sigma_{\lambda_{R}}$ the unique symmetric nonnegative definite matrix-valued function such that $\sgla^{2} = a- \lambda_{R}I$. then the process $X$ is strong Feller continuous.
 \end{prop}

 \begin{proof}  The same computations as those in the proof of Lemma \ref{lem-str-Feller} reveal that for each $R > 0$ and all $x, z\in \R^{d}$ with $|x| \vee |z| \le R$ and $0 < |x-z| \le \delta_{R}$, there exist positive constants   $ \delta_{R}$ and $\beta_{R}$ such that \begin{displaymath}
\wdh \LL F(|x-z|) \le -\beta_{R}  < 0.
\end{displaymath}  Use the same notations as those in the proof of Theorem \ref{thm-strong-Feller}.  For every $\e > 0$ and $t > 0$, we choose some $R > 0$ sufficiently large so that $\P(t > \tau_{R}) < \e$. For this chosen $R$, \eqref{eq-prob mean bdd by F}, in which the constant $\beta$ is replaced by $\beta_{R}$ and the stopping time $S_{\delta}$ replaced by $S_{\delta_{R}}$, remains valid. Now passing to limit as  $n\to \infty$ in  \eqref{eq-prob mean bdd by F} yields
\begin{displaymath}
F (\delta)\P  \set{T  \wedge \tau_{R}> S_{\delta_{R}}} +  \beta_{R} \E [T \wedge \tau_{R}\wedge S_{\delta_{R}} ] \le F (|x-z|).
\end{displaymath}
Then  for all $x, z\in \R^{d}$ with $|x| \vee |z| \le R$ and $0 < |x-z| \le \delta_{R}$, we can compute
\begin{align*}
\P\{T > t \} & = \P\{T > t, \tau_{R} \ge t, S_{\delta_{R}} > t \} +\P\{T > t, \tau_{R} \ge t, S_{\delta_{R}} \le t \} + \P\{T > t, \tau_{R} <  t \} \\
                & \le \P \{T \wedge \tau_{R} \wedge S_{\delta_{R}} > t \} + \P\{ T \wedge \tau_{R} > S_{\delta_{R}}\} + \e\\
                & \le  \biggl( \frac{1}{t\beta_{R}} + \frac{1}{F (\delta)}\biggr)  F  (|x-z|) + \e.
\end{align*} Consequently for any $f \in \B_{b}(\R^{d})$ and all $x, z\in \R^{d}$ with $|x| \vee |z| \le R$ and $0 < |x-z| \le \delta_{R}$, we have \begin{displaymath}
|P_{t} f(x) -P_{t}f(z)| \le  2\|f\|_{\infty}  \biggl( \frac{1}{t\beta_{R}} + \frac{1+\delta}{\delta}\biggr)|x-z| + 2\e\|f\|_{\infty}.
\end{displaymath} In particular, since $\e > 0$ is arbitrary, it follows that $\lim_{x-z \to 0}|P_{t} f(x) -P_{t}f(z)| = 0$; this gives the desired strong Feller property.   \end{proof}

\begin{rem}\label{rem-about str-Feller conditions}   Note that Assumption \ref{Assump-Uniform-Elliptic} places very mild condition on the function $\vartheta$. For instance,  when $c\equiv 0$, Theorem \ref{thm-strong-Feller} and Proposition \ref{prop-str-Feller-local version}  allow us to derive strong Feller property as long as the function  $b$ is locally uniformly continuous, and $\sigma_{\lambda_{0}}$  is locally H\"older continuous with
 exponent 
 $\delta_{\sigma_{\lambda_{0}}} > \frac12$.  
On the other hand, the uniform ellipticity condition for the diffusion matrix $a(x,k)$ in Assumption \ref{Assump-Uniform-Elliptic} is quite standard in the literature. Indeed,   similar assumptions are used in \cite{PriolaW-06, Qiao-14,Wang-10} to obtain the  strong Feller property. Proposition \ref{prop-str-Feller-local version} further relaxes this condition to a ``local'' one. In case that the diffusion matrix is degenerate, one needs to place certain conditions on the jumps to obtain strong Feller property; see \cite{Wang-11} for related work.
\end{rem}

\section{Irreducibility and Exponential Ergodicity}\label{sect: irreducibility and ergodicity}

The semigroup $P_{t}$ defined in \eqref{eq-semi-group} is said to be {\em irreducible} if for any $t > 0$ and $x\in \R^{d}$,
\begin{displaymath}
P_{t}(x,B) > 0 \text{ for all non-empty and open } B \subset \R^{d}.
\end{displaymath} A probability measure $\mu$ on $\R^{d}$ is said to be an {\em invariant measure} for the semigroup $P_{t}$ if $P_{t}^{*} \mu =\mu$ for all $t > 0$, where $P_{t}^{*} \mu(B): = \int_{\R^{d}} P_{t}(x,B)\mu(\d x)$, $B \in \B(\R^{d})$.

The following result improves Proposition  2.4 of \cite{Qiao-14}:

\begin{lem}\label{lem-irreducibility}
Suppose Assumption \ref{assumption-non-linear-growth-sec-2} $($with $\zeta\equiv 1)$    and Assumption \ref{assume-non-lip-sec-2} hold. Assume that there exists a constant $\lambda_{0} > 0$ such that \begin{equation}
\label{eq-diffusion-nondegenerate}
\lan y, a(x) y\ran \ge \lambda_{0} |y|^{2}, \quad\text{ for all }x,y\in \R^{d},
\end{equation} where $a(x) = \sigma(x)\sigma(x)^{T}$. Then the semigroup $P_{t}$ of \eqref{eq-semi-group}  is irreducible.
\end{lem}
\begin{rem}\label{rem-irreducibilty}
  Proposition  2.4  in \cite{Qiao-14} assumes slightly stronger conditions than those in Lemma \ref{lem-irreducibility}. In particular,   \cite{Qiao-14} assumes that \begin{displaymath}
2\lan x-y, b(x)-b(y) \ran + |\sigma(x)-\sigma(y)|^{2} + \int_{U}|c(x,u)-c(y,u)|^{2}\nu(\d u) \le K |x-y|^{2}\kappa(|x-y|),
\end{displaymath} for all $x,y\in \R^{d}$, where $K> 0$ and $\kappa $ is a positive and continuous function satisfying $\lim_{r\downarrow 0} \frac{\kappa(r)}{\log (r^{-1})} = \delta < \infty$. This condition excludes  functions such as  $r\mapsto \log (r^{-1}) \log(\log( r^{-1}))$ for $r > 0$ small.  By contrast,  Assumption \ref{assume-non-lip-sec-2} allows the modulus of continuity of the coefficients of \eqref{eq-sde-sec-2} to be of the form $r^{2} \log (r^{-2}) \log(\log( r^{-2}))$ for $r > 0$ small. Thanks to this relaxation,  the estimation techniques used in \cite{Qiao-14}   is not directly applicable in our analysis here. In addition, instead of requiring the modulus of continuity to hold for all $x,y\in \R^{d}$ as in \cite{Qiao-14},  Assumption \ref{assume-non-lip-sec-2}  only requires it in a small neighborhood of the diagonal line $x=y$ in $\{ (x,y) \in\R^{d}\times \R^{d}: |x| \vee |y| \le R\}$ for each $R>0$. Also, we note that the condition $\int_{U}|c(x,u)|^{4}\nu(\d u) \le K (1+|x|)^{4} $ in \cite{Qiao-14} is not necessary.
\end{rem}

Even though we  use essentially the same ideas of approximate controllability and Girsanov theorem as those in \cite{Qiao-14} and \cite{Zhang-09}, the technical difficulties arising  from the relaxed assumptions 
 merit a sketch of proof of Lemma  \ref{lem-irreducibility} in  Appendix \ref{sect-appendix}.

\begin{cor}\label{corollary-invariant-measure-uniqueness}
Under Assumptions \ref{assumption-non-linear-growth-sec-2}  $($with $\zeta\equiv 1)$, \ref{assume-non-lip-sec-2},    and \ref{Assump-Uniform-Elliptic}, then the semigroup $P_{t}$ of \eqref{eq-semi-group} has at most one invariant measure.
\end{cor}
\begin{proof} It is well known (see, for example, \cite{Cerrai-01}) that if a semigroup $P_{t}$ is irreducible and strong Feller, then it admits at most one invariant measure.  Under the stated assumptions, the semigroup $P_{t}$ is irreducible (by Lemma \ref{lem-irreducibility}) and strong Feller (by Theorem \ref{thm-strong-Feller}). Therefore the uniqueness of the invariant measure follows immediately.
\end{proof}
\begin{lem}\label{lem-invariant-measure-existence} Let Assumptions \ref{assumption-non-linear-growth-sec-2} and either \ref{assumption1-non-lip-sec-2} or \ref{assume-non-lip-sec-2}
 hold. Suppose there exist a positive constant $\alpha$, a compact $C\subset \R^{d}$, a measurable function $f: \R^{d}\mapsto [1,\infty)$, and twice continuously differentiable function $V: \R^{d}\mapsto \R_{+}$ satisfying \begin{equation}
\label{eq-pos-rec-condition}
\LL V(x) \le -\alpha f(x) + I_{C}(x), \text{ for all } x\in \R^{d}.
\end{equation}
Then the process $X$ of \eqref{eq-sde-sec-2} has an invariant measure.
\end{lem}

\begin{proof} This lemma can be proved using exactly the same arguments as those in the proof of Theorem 3.3 in \cite{Xi-04}. For brevity, we shall omit the details here.
\end{proof}

A combination of Corollary \ref{corollary-invariant-measure-uniqueness} and Lemma \ref{lem-invariant-measure-existence} yields the following proposition:

\begin{prop}\label{prop-invariant-measure-existence+uniqueness}
Under the assumptions of Corollary \ref{corollary-invariant-measure-uniqueness} and Lemma \ref{lem-invariant-measure-existence},  the semigroup $P_{t}$  of \eqref{eq-semi-group} has a unique invariant measure.
\end{prop}

For any positive function $f: \R^{d} \mapsto [1,\infty)$ and any signed measure $\nu$ defined on $\B(\R^{d} )$, we write
\begin{displaymath}
\|\nu\|_{f} : = \sup\{|\nu(g)|: g \in \B(\R^{d}) \text{ satisfying } |g| \le f \},
\end{displaymath} where $\nu(g) : = \int_{\R^{d}} g(x)\nu(\d x)$ is the integral of the function $g$ with respect to the measure $\nu$.
Note that the usual total variation norm $\|\nu\|_{\var}$ is just $\|\nu\|_{f}$ in the special case when $f\equiv 1$. For a function  $f:\R^{d}\mapsto [1, \infty)$, the process $X $ is said to be {\em $f$-exponentially ergodic}  if there exists a probability measure $\pi(\cdot) $, a constant $\theta  \in (0,1)$ and a finite-valued function $\Theta(x)$ such that
\begin{equation}
\label{eq-exp-ergodicity-defn}
\norm{P_{t}(x,\cdot) - \pi(\cdot)}_{f} \le \Theta(x) \theta^{t}
\end{equation}
for all $t\ge 0 $ and all $x \in \R^{d} $.

 \begin{thm}\label{thm-exp-ergodicity}
Suppose   Assumptions \ref{assumption-non-linear-growth-sec-2}  $($with $\zeta\equiv 1)$, \ref{assume-non-lip-sec-2},    and \ref{Assump-Uniform-Elliptic}  hold. In addition, assume that there exist positive numbers $\alpha, \beta$ and a nonnegative function $V\in C^{2}( \R^{d}) $ satisfying
\begin{itemize}
  \item[(i)] $V(x)\to \infty$ as $|x| \to \infty$,
    \item[(ii)] $\LL V(x) \le -\alpha V(x) + \beta$, $x \in  \R^{d} $.

\end{itemize} Then the process $X$ is $f$-exponentially ergodic with $f(x)= V(x) +1$.
\end{thm}

\begin{proof}
Apparently conditions (i) and (ii) in the statement of the theorem imply \eqref{eq-pos-rec-condition} and hence the existence and uniqueness of an invariant measure $\pi$ follows from Proposition \ref{prop-invariant-measure-existence+uniqueness}.   Next we can use the same argument as those in the proof of Theorem 6.3 in \cite{Xi-09} to obtain the desired $f$-exponential ergodicity for the process $X$.
\end{proof}

\begin{rem} Note that the condition  \begin{equation}
\label{eq-Qiao's condition}
2\lan x,b(x)\ran +| \sigma(x)|^{2} + \int_{U} |c(x,u)|^{2} \nu(\d u) \le -\lambda_{3} |x|^{r} + \lambda_{4},
\end{equation} in which $\lambda_{3} > 0, \lambda_{4} \ge 0$ and $r \ge 2$, in Theorem 1.3 of   \cite{Qiao-14} is a special case of the drift condition in Theorem \ref{thm-exp-ergodicity}.
Indeed, the left hand side of \eqref{eq-Qiao's condition} is just the infinitesimal generator $\LL$ applied to the function $V(x)= |x|^{2}$. And since $r \ge 2$, we can find positive constants $\alpha$ and $\beta $ so that $-\lambda_{3} |x|^{r} + \lambda_{4} \le -\alpha |x|^{2} + \beta = -\alpha V(x) + \beta$ for all $x\in \R^{d}$. In other words, \eqref{eq-Qiao's condition}  implies the drift condition of Theorem  \ref{thm-exp-ergodicity}.
\end{rem}

\begin{exm}\label{example2} Let us consider the following SDE \begin{equation}
\label{eq-example2}
\d X(t) = b(X(t))\d t + \sigma(X(t))\d W(t) + \int_{U} c(X(t-), u) \wdt N(\d t, \d u), X(0) = x \in \R^{3},
\end{equation} where $W$ is a 3-dimensional standard Brownian motion, $\wdt N(\d t,\d u)$ is a compensated Poisson random measure with compensator $\d t\,\nu(\d u)$ on $[0,\infty)\times U$, in which $U=  \{u\in \R^{3}: 0 < |u| <1 \}$ and $\nu(\d u) : = \frac{\d u}{|u|^{3+\alpha}}$ for some $\alpha\in (0,2)$. The coefficients of \eqref{eq-example2} are given by
\begin{displaymath}
b(x)=\begin{pmatrix}-x_{1}^{1/3}-x_{1}^{3}\\ -x_{2}^{1/3}-x_{2}^{3}\\-x_{3}^{1/3}-x_{3}^{3}\\ \end{pmatrix}, \quad
\sigma(x) = \begin{pmatrix} 3 & 1 & 2\\ 2 & 3 & 1 \\ 1 & 2 & 3\end{pmatrix},\quad
c(x,u) = \begin{pmatrix} \gamma x_{1}^{2/3} |u| \\  \gamma x_{2}^{2/3} |u|\\ \gamma x_{3}^{2/3} |u| \end{pmatrix},
\end{displaymath} in which $\gamma$ is a positive constant so that $\gamma^{2} \int_{U} |u|^{2} \nu(\d u) = \frac12$.

   We claim that all conditions in Theorem \ref{thm-exp-ergodicity} are satisfied and hence the process $X$ of  \eqref{eq-example2} is exponentially ergodic. Indeed,  detailed calculations similar to those in \eqref{eq1-ex1-non-lin-growth} and \eqref{eq2-ex1-non-Lip} help to verify Assumptions \ref{assumption-non-linear-growth-sec-2}  $($with $\zeta\equiv 1)$ and \ref{assume-non-lip-sec-2}.   On the other hand, it is clear that the matrix $a(x) = \sigma(x)\sigma(x)^{T}= \begin{pmatrix} 14 & 11 & 11 \\ 11 & 14 & 11\\ 11 & 11 & 14\end{pmatrix}$ is uniformly positive definite. Moreover, using similar  calculations as those in \eqref{eq2-ex1-non-Lip}, we can verify    condition \eqref{eq-str-feller-condition} and hence Assumption \ref{Assump-Uniform-Elliptic}. Finally we turn to the drift condition stated in Theorem \ref{thm-exp-ergodicity}. To this end, we consider the function $V(x) = |x|^{2}$, $x\in \R^{d}$, which clearly satisfies condition (i) in the statement of Theorem \ref{thm-exp-ergodicity}. On the other hand, straightforward calculations 
   lead to
\begin{align*}
\LL V(x) & = 2\lan x, b(x)\ran + |\sigma(x)|^{2} + \int_{U} |c(x,u)|^{2} \nu(\d u) = -\frac32\sum_{j=1}^{3}x_{j}^{4/3} -2\sum_{j=1}^{3}x_{j}^{4} + 42\\
    &  \le -2\sum_{j=1}^{3}x_{j}^{4} + 42 \le - \alpha \sum_{j=1}^{3}x_{j}^{2} + \beta = -\alpha V(x) + \beta,
\end{align*}for all $x\in \R^{d}$ and some positive constants $\alpha, \beta$.  This gives condition (ii)  in the statement of Theorem \ref{thm-exp-ergodicity} and hence the claimed exponential ergodicity.
\end{exm}

\section{Applications}\label{sect-applications}
\subsection{SDEs driven by L\'evy processes}\label{sect-sdes-Levy noise}\label{sect-Levy sde} 
We consider the stochastic differential equation
\begin{equation}
\label{eq0-Levy noise SDE}
\d X(t) = \psi(X(t-)) \d L(t), \quad X(0) = x \in \R^{d},
\end{equation} where the function $\psi: \R^{d}\mapsto \R^{d\times d}$ is Borel  measurable and   $L\in \R^{d}$ is a   L\'evy process  
with triplet $( b, Q, \nu)$. That is, $b \in \R^{d}, Q\in \R^{d\times d}$ is symmetric and nonnegative definite, and $\nu$ is a L\'evy measure on $\R_{0}^{d}:= \R^{d}\setminus\{0\}$ with $\int_{\R^{d}_{0}} 1\wedge |u|^{2} \nu(\d u) < \infty$.
It is well known that if $\psi$ is locally Lipschitz, then pathwise uniqueness holds for  \eqref{eq0-Levy noise SDE}. Thus our focus in this section is to
derive   non-Lipschitz conditions under which    pathwise uniqueness still holds for  \eqref{eq0-Levy noise SDE}.

Thanks to  the   L\'evy-It\^{o} decomposition theorem (see, for example, Theorem 2.4.16 of \cite{APPLEBAUM}), we can write $L$ as:
\begin{displaymath}
L(t) = bt  + \sigma W(t) + \int_{\R^{d}_{0}} u I_{\{ |u| \le 1\}} \wdt N(t, \d u) + \int_{\R^{d}_{0}} u I_{\{ |u| > 1\}}  N(t, \d u) ,
\end{displaymath} where 
$W \in \R^{d}$ is a standard Brownian motion,  and $\sigma \in \R^{d\times d}$ satisfies  $\sigma \sigma^{T} = Q$.
Using this  L\'evy-It\^{o} decomposition, we can rewrite \eqref{eq0-Levy noise SDE} as
\begin{equation}
\label{eq-SDE-Levy-noise}\begin{aligned}
\d X(t) = &\  \psi (X(t-)) b \d t  +  \psi(X(t-)) \sigma \d W(t)  \\&+  \int_{\{ |u| \le 1\}}  \psi(X(t-))  u \wdt N(\d t, \d u) +  \int_{\{ |u| > 1\}}  \psi(X(t-))  u N(\d t, \d u) .
\end{aligned}\end{equation}

\begin{prop}\label{prop-levy noise sde} 
The following assertions hold:
\begin{enumerate}
  \item[(i)] Suppose there  exist a positive constant $K$ and a nondecreasing and continuously differentiable function $\zeta:[0,\infty) \mapsto [1,\infty)$ satisfying  \eqref{eq-zeta-sec-2} such that \begin{equation}
\label{eq-psi-super linear-levy sde}
|\psi(x)|^{2} \le K(|x|^{2}\zeta(|x|^{2}) + 1), \quad \text{ for all }x\in \R^{d}.
\end{equation} Then the solution to \eqref{eq0-Levy noise SDE} has no finite explosion time a.s.
  \item[(ii)]  Suppose that there  exist   positive constants $\delta_{0}, K$ and   a nondecreasing, continuous and concave function $\varrho: [0,\infty) \mapsto [0, \infty)$ satisfying \eqref{eq-varrho-conditions-Feller} and $r \le K \varrho(r)$ for all $r \in [0,\delta_{0}]$ such that \begin{equation}\label{eq-psi-non-lip-levy sde}
|\psi(x) - \psi(z)|^{2} \le K \varrho(|x-z|^{2}), \quad \text{ for all }x,z\in \R^{d} \text{ with } |x-z| \le \delta_{0}.
\end{equation} Then pathwise uniqueness holds for \eqref{eq0-Levy noise SDE}.
\end{enumerate}
\end{prop}

  Some common functions satisfying the conditions of Proposition \ref{prop-levy noise sde} (ii) include $\varrho(r) = r, r\log(\frac1r), r\log(\log(\frac1r)),r\log(\frac1r)\log(\log(\frac1r)), \dots$ for $r $ in a small neighborhood $(0,\delta_{0}]$ of $0$.
\begin{proof}
These assertions follow directly  from applying Theorems \ref{thm-no-explosion-sec-2},  \ref{thm-path-uniqueness-sec-2},  and Corollary \ref{cor-existence+uniqueness general sde}      to \eqref{eq-SDE-Levy-noise},
  respectively.  For brevity, we shall omit the straightforward computations here.
  \end{proof}

Next we consider  sufficient conditions for Feller and strong Feller properties for the weak solution $X$ to \eqref{eq0-Levy noise SDE}.

\begin{prop}\label{prop-Levy sde Feller} Assume that the L\'evy measure $\nu$ also satisfies $\int_{\{|u| \ge 1\}} |u|\nu(\d u) < \infty$ and that \eqref{eq0-Levy noise SDE} has a  unique   non-exploding weak solution for every initial condition.  Suppose also that there exist  positive constants $K,\delta_{0}$ and  a nondecreasing and concave function $\varrho: [0,\infty) \mapsto [0,\infty)$ satisfying \eqref{eq-varrho-conditions-Feller} and $r \le K \theta(r)$ for all $r \in[0,\delta_{0}]$ such that \begin{equation}\label{eq-psi-non-Lip} |\psi(x) - \psi(z)|^{2}  \le K |x-z| \varrho (|x-z|),\quad \text{ for all }x,z\in \R^{d} \text{ with } |x-z| \le \delta_{0}, \end{equation} where $\delta_{0} > 0$.  Then the weak solution $X$ to \eqref{eq0-Levy noise SDE} is Feller continuous.
In addition, suppose there exists a positive number $\lambda_{0}$ such that \begin{equation}
\label{eq-Levy sde uniform elliptic}
\lan \xi, \psi(x)Q\psi(x)^{T}\xi\ran \ge \lambda_{0}|\xi|^{2}, \text{ for all }x,\xi\in \R^{d}.
\end{equation} Then the weak solution $X$ to \eqref{eq0-Levy noise SDE} is strong Feller continuous.
\end{prop}

\begin{proof} For the proof of Feller property, it is enough to verify that the coefficients of \eqref{eq-SDE-Levy-noise} satisfy Assumption \ref{assumption-feller-new condition}. Apparently \eqref{eq-psi-non-Lip} and the condition $r \le K \varrho (r)$ for all $r \in[0,\delta_{0}]$ imply that  $|x-z|  |\psi(x) - \psi(z)| \le K |x-z| \varrho (|x-z|)$ and hence
\begin{displaymath}
\lan x-z, (\psi(x) - \psi(z) )b\ran + |(\psi(x) - \psi(z) )\sigma|^{2} \le K |x-z| \varrho (|x-z|), \text{ for all } 
 |x-z| \le \delta_{0}.
\end{displaymath} On the other hand,
\begin{align*}
  \int_{\R_{0}^{d}}  &    |\psi(x) u - \psi(z)   u|^{2}  \wedge (4|x-z| |\psi(x)u - \psi(z)u|) \nu(\d u)  \\
    &   \le |\psi(x) -\psi(z)|^{2} \int_{\R_{0}^{d}}  |u|^{2} I_{\{ |u| \le 1\}}\nu(\d u) + 4|x-z| |\psi(x) - \psi(z)| \int_{\R_{0}^{d}}  |u| I_{\{ |u| > 1\}}\nu(\d u)\\
    & \le  K  |x-z| \varrho (|x-z|).
\end{align*} A combination of the above displayed equations gives \eqref{eq-feller-new condition} and hence verifies Assumption \ref{assumption-feller-new condition}. Then  we derive the Feller property for $X$   by Theorem \ref{thm-feller new}.

Concerning the strong Feller property, \eqref{eq-Levy sde uniform elliptic} and the calculations in the previous paragraph guarantee  that Assumption \ref{Assump-Uniform-Elliptic} is  satisfied and thus the desired strong Feller property holds true thanks to Theorem \ref{thm-strong-Feller}. \end{proof}

\subsection{L\'evy Type Operator and Feynman-Kac Formula}\label{sect-Feynman-Kac}
We consider the L\'evy type operator
\begin{equation}
\label{eq-Levy type operator}    \begin{aligned}
\mathcal{L}  f(x  ) & =  \frac{1}{2} \sum_{j,k=1}^{d}a_{jk}(x  ) \frac{\partial^{2}}{\partial x_{j}\partial x_{k}}  f(x  ) + \sum_{j=1}^{d} b_{j}(x  ) \frac{\partial }{\partial x_{j}} f(x  )    \\ 
 & \quad + \int_{\R^{d}_{0}} [f(x+y) - f(x  ) -  y\cdot Df(x  )] \nu(x  , \d y), \end{aligned}
\end{equation} in which $a(x) = (a_{jk}(x)) \in \R^{d\times d}$ is measurable, symmetric and nonnegative definite for all $x\in \R^{d}$,  $f\in C^{2}_{c}(\R^{d})$ and $\nu(x,\d y)$ is a L\'evy measure satisfying $\int_{\R^{d}_{0}} |y| \wedge |y|^{2} \nu(x,\d y) < \infty$ for all $x\in \R^{d}$. In addition, we assume that  there exist a positive constant $K$ and  a nondecreasing function $\zeta: [0,\infty) \mapsto [1,\infty)$ that is continuously differentiable and satisfies \eqref{eq-zeta-sec-2} so that
\begin{equation}
\label{eq-non-explosion-sect-5}
 2\lan x,b(x)\ran +\tr( a(x)) + \int_{\R_{0}^{d}} |y|^{2} \nu(x,\d y) \le K(|x|^{2}\zeta(|x|^{2}) + 1), \text{ for all }x \in \R^{d}.
\end{equation}

We wish to establish a Feynman-Kac formula for the solution to the Cauchy problem related to the L\'evy type operator $\LL$ of \eqref{eq-Levy type operator}:
\begin{equation}
\label{eq-Cauchy-terminal}
\left\{\!\!\begin{array}{rlll}
 \displaystyle  \frac{\partial }{\partial t} u(t,x) + \LL  u(t,x) - \rho(t,x) u(t,x) & \!\!=&\!\! g(t,x),     &  (t,x) \in [0,T) \times \R^{d}, \\
     u(T,x) &  \!\!=&\!\!  f(x) , &     x \in \R^{d},
\end{array}\right.
\end{equation} where
  the functions $\rho(\cdot,\cdot) \ge 0$, $g (\cdot,\cdot)$, and $f(\cdot)$ are continuous, and $ \LL  u(t,x)$ is interpreted as the operator $\LL$  applied to the function $x\mapsto u(t, x)$ and thus in particular, we require
    \begin{displaymath}
\int_{\R_{0}^{d}} |u(t,x+y) - u(t,x) - y\cdot D_{x} u(t, x)| \nu(x, \d y)  < \infty, \text{ for all } x \in \R^{d}.
\end{displaymath}

Let us first present the following lemma whose proof can be found in the Appendix \ref{sect-appendix}.
 \begin{lem}\label{lem-nu-representation}
  There exist  a measurable function   
  $c: \R^{d}  \times U \mapsto \R^{d}$  and a $\sigma$-finite measure ${M}$ on a measurable space $(U,\mathfrak U)$ such that
\begin{equation}
\label{eq-nu-measure-representation}
   \nu(x,  \Gamma) = \int_U I_{\Gamma}( c (x, u))   M(\d u), 
\end{equation}    for all   $ x\in \R^{d} $   and   $\Gamma \in \B(\R^{d}_{0})$. Consequently the  operator $\mathcal L$ of \eqref{eq-Levy type operator} can be rewritten as
\begin{equation}\label{eq-Lk-alternative-form}\begin{aligned}
 \mathcal{L} f(x)  =& \   \frac{1}{2} \sum_{j,k=1}^{d}a_{jk}(x) \frac{\partial^{2}}{\partial x_{j}\partial x_{k}}  f(x)  + \sum_{j=1}^{d} b_{j}(x) \frac{\partial }{\partial x_{j}} f(x) \\
 &  +   \int_{U} [f(x+c(x, u)) - f(x) -   c(x, u)\cdot Df(x)] M(\d u).
\end{aligned}\end{equation}
\end{lem}




Lemma \ref{lem-nu-representation}  now enables us to derive a stochastic differential equation corresponding to the L\'evy type operator $\LL$ of \eqref{eq-Levy type operator}. Indeed,
let $N$ be a Poisson random measure  on $U \times [0,\infty)$ with mean measure $\nu(\d u)\d t$ and denote its compensator measure by $\wdt N(\d u,\d t)= N(\d u, \d t)-\nu(\d u)\d t $.
 Let $\sigma: \R^{d} \mapsto \R^{d\times d } $ be a measurable  square root of $a$ so that $\sigma \sigma' (x) = a(x)$ for all $x \in \R^{d}\times \ss$.
Consider the following stochastic differential equation
 \begin{equation}\label{eq-X-sde-sect5}\begin{aligned}
   X(s) &   = x + \int_{t}^{s} b(X(s) ) \d s +
   \int_{t}^{s}  \sigma(X(s) ) \d W(s)   +  \int_{t}^{s}\!\int_{U} c (X(s-), u) \wdt N(\d u, \d s), \  s \ge t,
\end{aligned} \end{equation}
where $(t,x)\in[0,\infty)\times \R^{d}$ and $W$ is a standard $d$-dimensional Brownian motion.

 \begin{Assumption}\label{assumption-sect-5}  For any $(t,x)\in[0,\infty)\times \R^{d}$, the SDE \eqref{eq-X-sde-sect5} has a unique weak solution $((\Omega, \F, \P), \{\F_{s}\}_{s\ge t}, (W, N), X)$, in which  $(\Omega, \F, \P)$ is a probability space,  $\{\F_{s}\}_{s\ge t}$ is a filtration of $\F$ satisfying the usual condition, $W$ is an  $\{\F_{s}\}_{s\ge t}$-adapted Brownian motion, $N$ is an  $\{\F_{s}\}_{s\ge t}$-adapted Poisson random measure, and $X$ satisfies \eqref{eq-X-sde-sect5}. For simplicity, we   denote the weak solution by   $X= X^{t, x}$.   \end{Assumption}

 Note that Assumption \ref{assumption-sect-5} is equivalent to that the martingale problem for the  infinitesimal generator $\LL$ of \eqref{eq-Levy type operator} is well-posed for any initial condition $(t,x)\in [0,\infty) \times \R^{d}$;  see, for example, Theorem 2.3 of \cite{Kurtz-11}. We refer to \cite{Stroock-75} and \cite{Koma-73} for  investigations of the well-posedness of martingale problems for L\'evy type operators.

\begin{thm}\label{thm-FK-terminal}  Let Assumption \ref{assumption-sect-5} be satisfied.
Let $T >0$. Suppose that $u(\cdot,\cdot): [0,T]\times \R^{d}   \mapsto \R $ is of class $C^{1,2}([0,T) \times \R^{d}) \cap C_{b}([0,T]\times \R^{d})$  and satisfies
  the Cauchy  problem \eqref{eq-Cauchy-terminal}.
Assume that the functions $f, g$ are   uniformly bounded.
 Then we have 
\begin{equation}
\label{eq-FK-stoch-soln-terminal}
\begin{aligned}
u(t,x) =   \E_{t,x}   \biggl[ e^{-\int_{t}^{T} \rho(r, X(r)) dr } f(X(T))
-   \int_{t}^{T} \! e^{-\int_{t}^{s} \rho(r, X(r)) dr }  g(s,X(s)) \d s   \biggl], \   0\le t \le T, x\in\R^{d}.
\end{aligned}\end{equation}
\end{thm}

 \begin{proof}
 Thanks to Lemma \ref{lem-nu-representation}, we have \begin{equation}\label{eq-levy-sde-non-explosion-condition}  \int_{\R_{0}^{d}} |y|^{2} \nu(x,\d y)   = \int_{U} |c(x,u)|^{2} M(\d u). \end{equation} Putting this observation into  \eqref{eq-non-explosion-sect-5}, we see that the coefficients of \eqref{eq-X-sde-sect5}  satisfies Assumption \ref{assumption-non-linear-growth-sec-2}. Therefore  for any $(t,x)\in[0,\infty)\times \R^{d}$, Theorem \ref{thm-no-explosion-sec-2} implies  that the unique weak solution $X= X^{t,x}$ of \eqref{eq-X-sde-sect5} has no finite explosion time with probability 1.
We can then apply It\^o's formula to the process $e^{-\int_{t}^{s} \rho(r, X(r)) \d r} u(s, X(s)), s \in[t, T]$ and use the first equation of \eqref{eq-Cauchy-terminal} to see that
\begin{displaymath}
\xi(s; t,x) : = e^{-\int_{t}^{s} \rho(r, X(r)) \d r} u(s, X(s))- u(t, x) -\int_{t}^{s}  e^{-\int_{t}^{r} \rho(u, X(u)) \d u} g(r, X(r)) \d r, \ \  s\in[t,T]
\end{displaymath} is a local martingale. The boundedness assumptions on $u$ and $g$ in fact implies that $\xi$ is a bounded local martingale and hence a   martingale. In particular, we have $\E[\xi(T;t,x)] =0$, which, together with the terminal condition of  \eqref{eq-Cauchy-terminal}, leads to \eqref{eq-FK-stoch-soln-terminal}. This completes the proof.
\end{proof}

\begin{rem}
Note that in the traditional setting for Feynman-Kac formula, one typically imposes linear growth condition or boundedness condition on the coefficients $b,\sigma$ and $c$; see, for example, Theorem 5.7.6  of \cite{Karatzas-S} for the diffusion case and  Theorem 6.7.9 of \cite{APPLEBAUM} for the jump diffusion case. For our version of Feynman-Kac formula presented  in Theorem  \ref{thm-FK-terminal}, 
\eqref{eq-non-explosion-sect-5} allows the coefficients $b,\sigma$ and $c$ to grow super linearly. If  we also know that  $X$ has certain  moment estimates, say, $\E[\sup_{0\le s \le T} |X(s)|^{2}] < \infty$, then we can relax the boundedness assumption on $u, f, $ and $g$ to polynomial growth condition as in Theorem 3.2 of \citet*{ZhuYB-15}.
\end{rem}

\appendix
\section{Several Technical Proofs}\label{sect-appendix}

\begin{proof}[Proof of Theorem \ref{thm1-path-uniqueness-sec-2}] Thanks to the assumptions imposed on the function $\rho$, we can find a strictly decreasing sequence $\{a_{n}\}\subset (0, 1]$ with $a_{0} =1$, $\lim_{n\to\infty} a_{n} =0$ and $\int_{a_{n}} ^{a_{n-1}} \rho^{-1} (r) \d r = n$ for every $n \ge 1$. For each $n \ge 1$, there exists a continuous function $\rho_{n}$ on $\R$ with support in $(a_{n}, a_{n-1}) $ so that  $0 \le \rho_{n}(r) \le 2 n^{-1}\rho^{-1}(r) $ holds for every $r > 0$, and $\int_{a_{n}}^{a_{n-1}} \rho_{n}(r) \d r =1$.

Now consider the sequence of functions
\begin{equation}
\label{eq-psi-n-sec-2}
\psi_{n}(r) : = \int_{0}^{|r|}\int_{0}^{y} \rho_{n}(u) \d u \d y, \quad r\in \R, n \ge 1.
\end{equation}
We can immediately verify that $\psi_{n}$ is even and twice continuously differentiable, with $|\psi_{n}'(r)| \le 1$ and $\lim_{n\to\infty} \psi_{n}(r) = |r|$ for $r\in \R$. Furthermore, for each $r > 0$, the sequence $\{\psi_{n}(r) \}_{n\ge 1}$ is nondecreasing. Note also that for each $n\in \N$,  $\psi_{n}, \psi_{n}'$ and $\psi_{n}''$ all vanish  on the interval $(-a_{n}, a_{n})$.
By direct computations, we have for $0\neq x \in \R^{d}$
\begin{align*}
 D \psi_{n}(|x|) = \psi'_{n} (|x|) \frac{x}{|x|},\ \text{ and }\
 D^{2} \psi_{n}(|x|)  =\psi''_{n} (|x|) \frac{xx^{T}}{|x|^{2}} +  \psi'_{n} (|x|) \biggl[\frac{I}{|x|}- \frac{xx^{T}}{|x|^{3}}\biggr].
\end{align*}

Now suppose that $  X$ and $ \wdt X$ satisfy
\begin{align*}
 X(t)& =  x + \int_{0}^{t} b(X(s)) \d s + \int_{0}^{t}\sigma(X(s))\d W(s) + \int_{0}^{t} \int_{U} c(X(s-),u) \wdt N(\d s, \d u), \\
\wdt X(t)& = \wdt x + \int_{0}^{t} b( \wdt X(s)) \d s + \int_{0}^{t}\sigma( \wdt X(s))\d W(s) + \int_{0}^{t} \int_{U} c( \wdt X(s-),u) \wdt N(\d s, \d u),
\end{align*} for all $t \ge 0$, where $\wdt x, x \in \R^{d}$. Denote $\Delta _{t} : = \wdt X(t) - X(t)$ for $t \ge 0$. Assume $|\Delta_{0}| = |\wdt x-x| < \delta_{0} $ and define $$S_{\delta_{0}}: = \inf\{t \ge 0: |\Delta_{t}| \ge \delta_{0}\}= \inf\{t \ge 0: |\wdt X(t)- X(t)| \ge \delta_{0}\}.$$
 For $R >0$, let $\tau_{R}: =\inf\{t \ge 0: |\wdt X(t)| \vee |X(t)| > R \}$. By virtue of Theorem \ref{thm-no-explosion-sec-2}, we have $\tau_{R} \to \infty$ a.s. as $R\to \infty$.

Let us introduce the notations: \begin{align}
\label{eq-A-bar-notation}
 \lbar A(x,z)&: = \frac{|\lan x-z, \sigma(x) -\sigma(z)\ran|^{2}}{|x-z|^{2}}, \ \
 B(x,z): = \lan x-z, b(x)-b(z)\ran.
\end{align}
Applying It\^o's formula, we have
\begin{align}\label{eq-pathwise-uni-sec-2}
\nonumber \E&  [ \psi_{n} (|\Delta_{t\wedge \tau_{R}\wedge S_{\delta_{0}}} |)]
\\ &  = \psi_{n}(|\Delta _{0} |) + \E\biggl[\int_{0}^{t\wedge \tau_{R} \wedge S_{\delta_{0}}}I_{\{\Delta_{s} \neq0\}}\biggl[ \frac12\biggl( \psi_{n}''(|\Delta_{s}| ) - \frac{\psi_{n}'(|\Delta_{s}|)} { | \Delta_{s}| }   \biggr) \lbar A(\wdt X(s), X(s)) \\
\nonumber &\qquad \qquad\qquad+ \frac{\psi_{n}'(|\Delta_{s}|)} {2 | \Delta_{s}| } \bigl( 2B(\wdt X(s), X(s))  +  |\sigma(\wdt X(s)) - \sigma(X(s))|^{2} \bigr)\biggr] \d s \\
\nonumber  & \qquad \qquad\qquad+ \int_{0}^{t\wedge \tau_{R}\wedge S_{\delta_{0}}} \int_{U} \biggl[ \psi_{n}(|\Delta_{s} + c(\wdt X(s-), u) - c(X(s-),u)|) - \psi_{n}(|\Delta_{s}|) \\
\nonumber   &\qquad \qquad\qquad -I_{\{\Delta_{s} \neq0\}} \frac{\psi_{n}'(|\Delta_{s}|)}{|\Delta_{s}|} \lan \Delta_{s}, c(\wdt X(s-), u) - c(X(s-),u)\ran\biggr] \nu(\d u) \d s\biggr].
\end{align}
Recall that we have $0 \le \psi_{n}'(r) \le 1$ for each $r \ge 0$. Thus it follows from \eqref{eq-coeffs-non-lip-sec-2} that
 \begin{align}
\label{eq1-pathwise-uni-sec-2}
\nonumber \E&\biggl[\int_{0}^{t\wedge \tau_{R} \wedge S_{\delta_{0}}} I_{\{\Delta_{s} \neq0\}}\frac{\psi_{n}'(|\Delta_{s}|)} {2 | \Delta_{s}| } \bigl( 2B(\wdt X(s), X(s))  +  |\sigma(\wdt X(s)) - \sigma(X(s))|^{2} \bigr)  \d s  \biggr]  \\
\nonumber  & \le  \E\biggl[\int_{0}^{t\wedge \tau_{R} \wedge S_{\delta_{0}}} I_{\{\Delta_{s} \neq0\}}\frac{\psi_{n}'(|\Delta_{s}|)} {2 | \Delta_{s}| }  { \kappa_{R}}|\Delta_{s}| \rho(|\Delta_{s}|)\d s\biggr]\\
 & \le \E\biggl[\int_{0}^{t\wedge \tau_{R} \wedge S_{\delta_{0}}}  \frac{{ \kappa_{R}} }{2}\rho(|\Delta_{s}|)\d s  \biggr] = \frac{{ \kappa_{R}} }{2} \E\biggl[\int_{0}^{t\wedge \tau_{R}\wedge S_{\delta_{0}}}  \rho(|\Delta_{s}|)\d s  \biggr].
\end{align}   
On the other hand, thanks to the construction of $\psi_{n}$, we have for all $r\ge 0$,
$\psi_{n}''(r) = \rho_{n}(r) \le \frac{2}{n\rho(r)} I_{(a_{n},a_{n-1})}(r).$
Then it follows from \eqref{eq-coeffs-non-lip-sec-2} that
\begin{align}
\label{eq2-pathwise-uni-sec}
\nonumber \E&\biggl[\int_{0}^{t\wedge \tau_{R}\wedge S_{\delta_{0}}}  \frac12I_{\{\Delta_{s} \neq0\}}\biggl( \psi_{n}''(|\Delta_{s}| ) - \frac{\psi_{n}'(|\Delta_{s}|)} { | \Delta_{s}| }   \biggr) \lbar A(\wdt X(s), X(s)) \d s\biggr] \\
\nonumber  &\le \frac12 \E\biggl[\int_{0}^{t\wedge \tau_{R}\wedge S_{\delta_{0}}}I_{\{\Delta_{s} \neq0\}} \psi_{n}''(|\Delta_{s}| ) \lbar A(\wdt X(s), X(s)) \d s\biggr] \\
\nonumber &  \le  \frac12 \E\biggl[\int_{0}^{t\wedge \tau_{R}\wedge S_{\delta_{0}}}  \frac{2}{n\rho(|\Delta_{s}|)} I_{(a_{n},a_{n-1})} (|\Delta_{s}|) \frac{|\Delta_{s}|^{2} |\sigma(\wdt X(s))- \sigma(X(s))|^{2}}{|\Delta_{s}|^{2}} \d s\biggr]  \\
 & =   \E\biggl[\int_{0}^{t\wedge \tau_{R}\wedge S_{\delta_{0}}}  \frac{{ \kappa_{R}} }{n\rho(|\Delta_{s}|)} I_{(a_{n},a_{n-1})} (|\Delta_{s}|) |\Delta_{s}| \rho(|\Delta_{s}|) \d s\biggr]
 \le  \frac{ { \kappa_{R}}  t a_{n-1}}{n}.      
\end{align}
Using \eqref{eq-integral-term-non-lip-sec-2} and the fact that $|\psi_{n}'(r)| \le 1$, we can compute
  \begin{align}
\label{eq3-pathwise-uni-sec-2}
\nonumber \E&\biggl[ \int_{0}^{t\wedge \tau_{R}\wedge S_{\delta_{0}}} \int_{U} \biggl( \psi_{n}(|\Delta_{s} + c(\wdt X(s-), u) - c(X(s-),u)|) - \psi_{n}(|\Delta_{s}|)  \\
\nonumber & \qquad \qquad  \quad
 -I_{\{\Delta_{s} \neq0\}} \frac{\psi_{n}'(|\Delta_{s}|)}{|\Delta_{s}|} \lan \Delta_{s}, c(\wdt X(s-), u) - c(X(s-),u)\ran\biggr) \nu(\d u) \d s\biggr]\\
\nonumber    &  \le 2 \E \biggl[ \int_{0}^{t\wedge \tau_{R}\wedge S_{\delta_{0}}} \int_{U}  |c(\wdt X(s-), u) - c(X(s-),u)| \nu(\d u) \d s \biggr] \\
    &  \le 2 { \kappa_{R}}  \E \biggl[ \int_{0}^{t\wedge \tau_{R}\wedge S_{\delta_{0}}} \rho(|\Delta_{s}|) \d s\biggr].
\end{align} Plugging  \eqref{eq1-pathwise-uni-sec-2}--\eqref{eq3-pathwise-uni-sec-2} into \eqref{eq-pathwise-uni-sec-2}, we obtain
\begin{align*}
\E [ \psi_{n} (|\Delta_{t\wedge \tau_{R}\wedge S_{\delta_{0}}} |)] \le  \psi_{n}(|\Delta _{0} |) +   \frac{{ \kappa_{R}} t a_{n-1}}{n} + 
 \frac{5 { \kappa_{R}} }{2}\E\biggl[\int_{0}^{t\wedge\tau_{R}\wedge S_{\delta_{0}}} \rho(|\Delta_{s}|) \d s\biggr].
\end{align*}
Letting $n\to \infty$ yields \begin{align*}
\E [ |\Delta_{t\wedge \tau_{R}\wedge S_{\delta_{0}}} |] & \le |\Delta _{0} | +    \frac{5 { \kappa_{R}} }{2}\E\biggl[\int_{0}^{t\wedge\tau_{R}\wedge S_{\delta_{0}}} \rho(|\Delta_{s}|) \d s\biggr]   \le  |\Delta _{0} | +    \frac{5 { \kappa_{R}} }{2}\E\biggl[\int_{0}^{t} \rho(|\Delta_{s\wedge \tau_{R} \wedge S_{\delta_{0}}}|) \d s\biggr] \\
& \le  |\Delta _{0} | +    \frac{5 { \kappa_{R}} }{2} \int_{0}^{t}  \rho(\E[|\Delta_{s\wedge \tau_{R} \wedge S_{\delta_{0}}}|]) \d s,
\end{align*}
where we used the concavity of $\rho$ and Jensen's inequality to derive the last inequality. Let $u(t): = \E[ |\Delta_{t\wedge \tau_{R}\wedge S_{\delta_{0}}}|]$. Then $u$ satisfies $$0 \le u(t) \le v(t): = |\Delta _{0} | +    \frac{5\kappa_{R}}{2}  \int_{0}^{t} \rho(u(s)) \d s.$$

 Define $G(r):= \int_{1}^{r} \frac{\d s}{\rho(s)}$ for $r > 0$. Then   $G$ is nondecreasing and satisfies $G(r) > -\infty$ for $r > 0$ and $\lim_{r\downarrow 0} G(r) = -\infty$ thanks to \eqref{eq-rho-sec-2}. In addition, we have
 \begin{align*}
G(u(t)) & \le       G(v(t)) = G(|\Delta _{0} |) + \int_{0}^{t} G'(v(s)) v'(s) \d s \\
         & = G(|\Delta _{0} |) + \frac{5\kappa_{R}}{2}  \int_{0}^{t}  \frac{  \rho(u(s))}{\rho(v(s))}\d s
      \le G(|\Delta _{0} |) + \frac{5\kappa_{R}}{2}  t,
\end{align*} where the last inequality follows from the assumption that $\rho $ is nondecreasing. Now sending $|\Delta _{0} |= |\wdt x-x| \to 0$, we see that the right-hand side of the above inequality converges to $-\infty$ and so does the left-hand side. Hence\begin{equation}
\label{eq-u(t) to 0}
\lim_{|\wdt x - x| \to 0} u(t)=\lim_{|\wdt x - x| \to 0}\E[ |\Delta_{t\wedge \tau_{R}\wedge S_{\delta_{0}}}|]=0.
\end{equation}
In particular, when $\wdt x = x$, we have
  $\E [ |\Delta_{t\wedge \tau_{R}\wedge S_{\delta_{0}} } |] =0$. Recall that $\lim_{R\to \infty}\tau_{R}= \infty$ a.s.  Thus by Fatou's lemma,   we have $0 \le \E[ |\Delta_{t\wedge  S_{\delta_{0}} } |] \le \lim_{R\to\infty} \E[ |\Delta_{t\wedge \tau_{R}\wedge S_{\delta_{0}} } |] =0$. This gives $\E[ |\Delta_{t\wedge  S_{\delta_{0}} } |] =0$ and therefore $\Delta_{t\wedge S_{\delta_{0}} } =0$ a.s.

  On the set $\{S_{\delta_{0}} \le t\}$, we have  $|\Delta_{t\wedge S_{\delta_{0}} } | \ge \delta_{0}$. Thus it follows that $0 =\E [ |\Delta_{t\wedge S_{\delta_{0}} } |] \ge \delta_{0} \P\{ S_{\delta_{0}} \le t\}$. Then, we have $ \P\{ S_{\delta_{0}} \le t\} =0$ and hence $\Delta_{t} =0$ a.s.
 The desired pathwise uniqueness result then follows from the fact that $\wdt X$ and $X$ have right continuous sample paths.
\end{proof}

\begin{proof}[Proof of Theorem \ref{thm-path-uniqueness-sec-2}] Let $X(t),\wdt X(t), \Delta_{t}$, $S_{\delta_{0}}$,  and $\tau_{R}$ be defined as in the proof of Theorem \ref{thm1-path-uniqueness-sec-2}.
Consider the function $H(r): = \frac{r^{2}}{1+ r^{2}}$, $r\in \R$. We have $
H'(r) = \frac{2 r}{(1+r^{2})^{2} }$ 
and $ H''(r) = \frac{2 }{(1+r^{2})^{2} }  - \frac{8r^{2} }{(1+r^{2})^{3} }.$
Note that $H, H'$ and $H''$ are uniformly bounded.
By direct computations, we have for all $ x \in \R^{d}$
\begin{align*}
D H (|x|) &  
= \frac{2x}{(1+|x|^{2})^{2}},\ \text{ and }\
   D^{2} H(|x|)  
    =  \frac{2I}{(1+|x|^{2})^{2}} - \frac{8xx^{T}}{(1+|x|^{2})^{3}}.
\end{align*}
Applying It\^o's formula to the process $H(|\Delta_{\cdot\wedge \tau_{R} \wedge S_{\delta_{0}}}|)$, we have
\begin{align}\label{eq2-pathwise-uni-sec-2}
\nonumber  &\E [ H  (|\Delta_{t\wedge \tau_{R} \wedge S_{\delta_{0}}} |)] \\
\nonumber  & \ = H (|\Delta _{0} |) + \E\biggl[\int_{0}^{t\wedge \tau_{R} \wedge S_{\delta_{0}} } \biggl[ \frac{\lan 2\Delta_{s}, b(\wdt X(s))-b(X(s))\ran}{(1+|\Delta_{s}|^{2})^{2}}   \\
\nonumber   & \ \ + \frac12\tr\biggl((\sigma( \wdt X(s))- \sigma(X(s)))(\sigma( \wdt X(s))- \sigma(X(s)))^{T}\biggl (\frac{2I}{(1+|\Delta_{s}|^{2})^{2}} - \frac{8\Delta_{s} \Delta_{s}^{T}}{(1+|\Delta_{s}|^{2})^{3}}\biggr)\biggr) \biggr] \d s \\
\nonumber  & \ \ + \int_{0}^{t\wedge \tau_{R} \wedge S_{\delta_{0}}} \int_{U} \biggl[ H (|\Delta_{s} + c(\wdt X(s-), u) - c(X(s-),u)|) - H (|\Delta_{s}|) \\
\nonumber   & \qquad \quad\quad - \frac{2}{(1+|\Delta_{s}|^{2})^{2}} \lan \Delta_{s}, c(\wdt X(s-), u) - c(X(s-),u)\ran\biggr] \nu(\d u) \d s\biggr]\\
  &\ \le H (|\Delta _{0} |) + \E\biggl[\int_{0}^{t\wedge \tau_{R} \wedge S_{\delta_{0}}}  \frac{2\lan \Delta_{s}, b(\wdt X(s))-b(X(s))\ran  +  |\sigma(\wdt X(s)) - \sigma(X(s))|^{2}  }{(1+|\Delta_{s}|^{2})^{2}}\\
\nonumber   & \ \ + \int_{0}^{t\wedge \tau_{R} \wedge S_{\delta_{0}}} \int_{U} \biggl( H (|\Delta_{s} + c(\wdt X(s-), u) - c(X(s-),u)|) - H (|\Delta_{s}|) \\
\nonumber   & \qquad \quad\quad - \frac{2}{(1+|\Delta_{s}|^{2})^{2}} \lan \Delta_{s}, c(\wdt X(s-), u) - c(X(s-),u)\ran\biggr)\nu(\d u) \d s\biggr].
\end{align}

To simplify notations, for any $x,z\in \R^{d}$ and $u\in U$, let us denote $w: = w(x,z,u) = c(x, u) - c(z,u)$.
  Then
\begin{align*}
 H&  (|x+c(x,u)-z-c(z,u)|) - H (|x-z|) - \frac{2}{(1+|x-z|^{2})^{2}}  \lan x-z, c(x,u) - c(z,u)\ran          \\
      &   =  H(|x-z+w|) - H (|x-z|)   - \frac{H '(|x-z|)}{|x-z|}  \lan x-z, w\ran\\
    & = \frac{|x-z+w|^{2}}{1+|x-z+w|^{2}}- \frac{|x-z|^{2}}{1+|x-z|^{2}} - \frac{2\lan x-z, w\ran}{(1+|x-z|^{2})^{2}}\\
    & = \frac{|x-z+w|^{2}-|x-z|^{2}}{(1+|x-z+w|^{2})(1+|x-z|^{2})}- \frac{|x-z+w|^{2}-|x-z|^{2}}{ (1+|x-z|^{2})^{2}}\\
    & \quad + \frac{|x-z+w|^{2}-|x-z|^{2}}{ (1+|x-z|^{2})^{2}}- \frac{2\lan x-z, w\ran}{(1+|x-z|^{2})^{2}}\\
    & = \frac{|x-z+w|^{2}-|x-z|^{2}}{ 1+|x-z|^{2}}\biggl[\frac{1}{1+|x-z+w|^{2}} -  \frac{1}{1+|x-z|^{2}}\biggr] + \frac{|w|^{2}}{(1+|x-z|^{2})^{2}}\\
    & \le  \frac{|w|^{2}}{(1+|x-z|^{2})^{2}}. 
\end{align*}
Then we have
\begin{align*} 
 \int_{U} & \biggl( H (|\Delta_{s} + c(\wdt X(s-), u) - c(X(s-),u)|) - H (|\Delta_{s}|)  \\
&\pushright{   - \frac{2\lan \Delta_{s}, c(\wdt X(s-), u) - c(X(s-),u)\ran}{(1+|\Delta_{s}|^{2})^{2}} \biggr)\nu(\d u) }
\\ &   \le   \int_{U}   \frac{| c(\wdt X(s-), u) - c(X(s-),u) |^{2} 
  }{(1+|\Delta_{s}|^{2})^{2}}\nu(\d u).
 \end{align*} Using this estimate in \eqref{eq2-pathwise-uni-sec-2}, we obtain
  \begin{align*}
  \E& [ H  (|\Delta_{t\wedge \tau_{R} \wedge S_{\delta_{0}}} |)] -  H (|\Delta _{0} |) \\
  & \le  \E\biggl[\int_{0}^{t\wedge \tau_{R} \wedge S_{\delta_{0}}}   \frac{1}{(1+|\Delta_{s}|^{2})^{2}}  \biggl(  2 \lan \Delta_{s}, b(\wdt X(s))-b(X(s))\ran  +  |\sigma(\wdt X(s)) - \sigma(X(s))|^{2}   \\ &\pushright{ \hfill
 +  \int_{U}   | c(\wdt X(s-), u) - c(X(s-),u) |^{2} 
  \nu(\d u)  \biggr)  \d s\biggr].}
\end{align*}
Then, thanks to \eqref{eq-Feller-condition} and the first condition of \eqref{eq-varrho-conditions-Feller}, it follows that \begin{align*}
  \E[ H  (|\Delta_{t\wedge \tau_{R}\wedge S_{\delta_{0}}} |)] &\le  H (|\Delta _{0} |)  +\E\biggl[\int_{0}^{t\wedge \tau_{R}\wedge S_{\delta_{0}}}   \frac{{ \kappa_{R}} \varrho(|\Delta_{s}|^{2})}{(1+|\Delta_{s}|^{2})^{2}} \d s\biggr] \\
   &\le  H (|\Delta _{0} |)  +{ \kappa_{R}}\E\biggl[\int_{0}^{t\wedge \tau_{R}\wedge S_{\delta_{0}}}\varrho\biggl(\frac{|\Delta_{s}|^{2}}{ 1+|\Delta_{s}|^{2} }\biggr)\d s\biggr] \\
  & \le  H (|\Delta _{0} |)  +{ \kappa_{R}}\E\biggl[ \int_{0}^{t} \varrho( H(|\Delta_{s\wedge \tau_{R}\wedge S_{\delta_{0}}}|))\d s\biggr] \\
  & \le H (|\Delta _{0} |)  +{ \kappa_{R}}  \int_{0}^{t} \varrho(\E[ H(|\Delta_{s\wedge \tau_{R}\wedge S_{\delta_{0}}}|)])\d s.
  \end{align*}
   where we used the concavity of $\rho$ and Jensen's inequality to derive the last inequality.
When $\wdt x = x$ or $\Delta_{0} =0$, the 
same argument as that in the end of the proof of Theorem \ref{thm1-path-uniqueness-sec-2} reveals that $\E [H( |\Delta_{t \wedge\tau_{R}\wedge S_{\delta_{0}}} |)] =0$.  Since $\lim_{R\to \infty}\tau_{R}= \infty$ a.s. and $0 \le H(r) \le 1$ for all $r\ge 0$,
the bounded convergence theorem further implies that  $\E [H( |\Delta_{t  \wedge S_{\delta_{0}}} |)] =0$.

On the set $\{S_{\delta_{0}} < t\}$, $|\Delta_{S_{\delta_{0}}}| \ge \delta_{0}$. Since $H$ is increasing on $(0, \infty)$ and bounded above by $1$, it follows that $0 < H(\delta_{0}) \le H(|\Delta_{S_{\delta_{0}}}|) \le 1$ and hence $$H(\delta_{0})\P\{S_{\delta_{0}} < t\} \le \E[H( |\Delta_{t \wedge S_{\delta_{0}}} |) I_{\{S_{\delta_{0}} < t \}}] \le \E [H( |\Delta_{t \wedge S_{\delta_{0}}} |)] =0 .$$ Therefore it follows that $\P\{S_{\delta_{0}} < t\} =0$. Then $0 \le \E [H(|\Delta_{S_{\delta_{0}}}|)I_{\{  S_{\delta_{0}}< t\}} ]  \le   \E [1\cdot I_{\{  S_{\delta_{0}}< t\}} ] =0 $ and thus
\begin{align*}
 0=    \E\big[H( |\Delta_{t \wedge S_{\delta_{0}}} |)\big]  = \E\big[H(|\Delta_{t}|)I_{\{ t \le S_{\delta_{0}}\}}\big] +   \E\big[H(|\Delta_{S_{\delta_{0}}}|)I_{\{  S_{\delta_{0}}< t\}}\big] =   \E\big[H(|\Delta_{t}|)I_{\{ t \le S_{\delta_{0}}\}}\big].
\end{align*} Next we observe that
\begin{align*}
 \big|\E[H(|\Delta_{t}|)] -\E [H( |\Delta_{t \wedge S_{\delta_{0}}} |) ]  \big| &=  \big|\E[H(|\Delta_{t}|)] - \E [H(|\Delta_{t}|)I_{\{ t \le S_{\delta_{0}}\}} ] \big| \\ & =  \big|\E[H(|\Delta_{t}|)I_{\{  S_{\delta_{0}} < t\}} ]  \big| \le \P\{  S_{\delta_{0}} < t\} =0.
  \end{align*} Hence it holds that $\E[H(|\Delta_{t}|)] =0$  and hence $\Delta_{t}=0$ a.s.
As observed in the end of the proof of Theorem \ref{thm1-path-uniqueness-sec-2}, this  gives  the desired pathwise uniqueness result.
\end{proof}

\begin{proof}[Proof of Lemma \ref{lem-wdt-LL-F estimate}]
We have $F'(r) =  \frac{1}{(1+r)^{2}}$ and $F''(r) = -\frac{2}{(1+r)^{3}}$.
Recall the notations $\lbar A(x,z)$ and $B(x,z)$ defined in \eqref{eq-A-bar-notation}.
Then as in the proof of Theorem 3.1 in \cite{ChenLi-89}, straightforward calculations lead to
\begin{align}
\nonumber  \wdt \Omega_{\text{diffusion}}F(|x-z|)   & =  \frac{ F''(|x-z|)}{2}\overline A (x,z)
+ \frac{ F'(|x-z|) }{2|x-z|} \big[  |\sigma(x) - \sigma(z)|^{2}-\overline A(x,z) + 2  B(x,z)\big]\\
 \label{eq-Omega diffusion estimate}& \le \frac{ |\sigma(x) - \sigma(z)|^{2}  + 2  B(x,z)}{ 2 |x-z| (1+ |x-z|)^{2}}.
\end{align}
 Following the same arguments as those in the proof of Proposition 3.1  in \cite{Wang-10}, we can verify that
\begin{align}
\nonumber F&(|x+c(x,u)-z-c(z,u)|) - F(|x-z|)
    \\ \label{eq1-F estimate}&\quad - \frac{F'(|x-z|)}{|x-z|}  \lan x-z, c(x,u) - c(z,u)\ran \le \frac{|c(x,u)-c(z,u)|^{2}}{2 |x-z|(1+|x-z|)^{2}}. \end{align}
On the other hand, since the function $F$ is concave, it follows that $F(r) - F(r_{0}) \le F'(r_{0}) (r-r_{0})$ for all $r, r_{0}\ge 0$. Using this inequality with $r_{0} = |x-z|$ and $r = |x+c(x,u)-z-c(z,u)|$, and noting that $F'(r_{0}) > 0$, we can compute
\begin{align}
\nonumber F&(|x+c(x,u)-z-c(z,u)|) - F(|x-z|) - \frac{F'(|x-z|)}{|x-z|}  \lan x-z, c(x,u) - c(z,u)\ran          \\
 \nonumber   &  \le  F'(|x-z|) (|x+c(x,u)-z-c(z,u)|-|x-z| )- \frac{F'(|x-z|)}{|x-z|}  \lan x-z, c(x,u) - c(z,u)\ran \\
 \nonumber & \le  F'(|x-z|) | c(x,u) - c(z,u)| + \frac{F'(|x-z|)}{|x-z|}| x-z| \cdot| c(x,u) - c(z,u)| \\
 \nonumber & = 2 F'(|x-z|) | c(x,u) - c(z,u)| \\
  \label{eq2-F estimate}  & = \frac{2 |x-z|  |c(x,u) -c(z,u)|} {|x-z|(1+ |x-z|)^{2}}. 
\end{align} 
Combining \eqref{eq1-F estimate} and  \eqref{eq2-F estimate} yields \begin{align}
\nonumber F&(|x+c(x,u)-z-c(z,u)|) - F(|x-z|) - \frac{F'(|x-z|)}{|x-z|}  \lan x-z, c(x,u) - c(z,u)\ran \\ \label{eq-F estimate} &    \pushright{\le \frac{1}{2|x-z|(1+ |x-z|)^{2}} \bigl[|c(x,u)-c(z,u)|^{2} \wedge (4 |x-z| |c(x,u) -c(z,u)|)\bigr].}
\end{align}
Using \eqref{eq-F estimate} in $\wdt {\Omega}_{\text{jump}} $ of \eqref{eq-Omega-j-defn},  it follows that for all $x\neq z$,
\begin{align}
   \nonumber&  \wdt {\Omega}_{\text{jump}}  F(|x-z|) \\
\nonumber &   = \int_{U} \biggl[F (|x+c(x,u)-z-c(z,u)|) - F(|x-z|) \\
 \nonumber&   \qquad \qquad - \frac{F'(|x-z|)}{|x-z|}  \lan x-z, c(x,u) - c(z,u)\ran\biggr]  \nu(\d u) \\
   \label{eq-Omega jump F estimate} &   \le \frac{1}{2|x-z|(1+ |x-z|)^{2}}  \int_{U}\bigl[|c(x,u)-c(z,u)|^{2} \wedge (4 |x-z| |c(x,u) -c(z,u)|)\bigr]\nu(\d u).
\end{align}
Combining \eqref{eq-Omega diffusion estimate}  and \eqref{eq-Omega jump F estimate}, and using condition \eqref{eq-feller-new condition}, we obtain $$\wdt \LL F(|x-z|) \le  \frac{\kappa_{R} \varrho (|x-z|) }{ (1+ |x-z|)^{2}}  \le \kappa_{R} \varrho \biggl(\frac{|x-z|}{1+|x-z|}\biggr)= \kappa_{R}\varrho (F(|x-z|)), $$
for all $x,z\in \R^{d}$ with $|x| \vee |z| \le R $ and $0 < |x-z| \le \delta_{0}$, where we used \eqref{eq-varrho-conditions-Feller} to derive the second inequality above.
This establishes  \eqref{eq-wdt-LL-F estimate}  and hence completes the proof of the lemma.
\end{proof}

\begin{proof}[Proof of Lemma \ref{lem-irreducibility}]  Let us fix $T > 0, r > 0$ and $x,a\in \R^{d}$. We need to show that $P_{t}(x, B(a,r)) : = \P\{|X^{x}(T) -a | \le r\} > 0$, or equivalently, $ \P\{|X^{x}(T) -a | > r\} < 1$. To this end, we choose $t_{0} \in (0, T)$, whose exact value will be specified later. Set for $n \in \N$, $X^{n}(t_{0}) : = X(t_{0})  I_{\{ |X(t_{0})| \le n\}}$. Then we have \begin{equation}
\label{eq1-irreducibility-proof}
\lim_{n\to \infty} \E[H(|X(t_{0}) -X^{n}(t_{0})|)] =0,
\end{equation} where the function $H(r)=\frac{r^{2}}{1+ r^{2}}, r\ge 0$ is defined in the proof of Theorem \ref{thm-path-uniqueness-sec-2}.

For $t \in [t_{0}, T]$, we define $$ J^{n}(t) : = \frac{T-t}{T-t_{0}} X^{n}(t_{0}) + \frac{t-t_{0}}{T-t_{0}} a, \text{ and }h^{n}(t): = \frac{a-X^{n}(t_{0})}{T-t_{0}} - b(J^{n}(t)). $$ Then $J^{n}(t_{0}) =X^{n}(t_{0}), J^{n}(T) =a$, and $J^{n}$ satisfies the following SDE:
\begin{displaymath}
J^{n}(t)  = X^{n}(t_{0}) + \int_{t_{0}}^{t} b(J^{n}(s) )\d s +  \int_{t_{0}}^{t} h^{n}(s) \d s, \qquad t \in  [t_{0}, T].
\end{displaymath} Let us also consider the SDE
\begin{displaymath}
Y(t) : = X(t_{0})  + \int_{t_{0}}^{t} [b(Y(s)) +   h^{n}(s)] \d s +  \int_{t_{0}}^{t}\sigma(Y(s))\d W(s) + \int_{t_{0}}^{t}\int_{U} c(X(s-),u)\wdt N(\d s \d u),
\end{displaymath} for $t \in [t_{0}, T]$. Also let $Y(t) : = X(t)$ for $t \in [0, t_{0}]$. Denote $\Delta_{t}: = Y(t) - J^{n}(t)$ for $t\in [t_{0}, T]$. Note that $\Delta_{t_{0}} = X(t_{0}) -  X^{n}(t_{0})$ and $\Delta_{T} = Y(T) -a$.

Define $\tau_{R}: = \inf\{t \ge t_{0}: |Y(t)|\vee| J^{n}(t)| > R \} \wedge T$ and $S_{\delta_{0}}: = \inf\{t \ge t_{0}: |Y(t) - J^{n}(t)| \ge \delta_{0}\}\wedge T$. Then detailed calculations as those in the proof of Theorem \ref{thm-path-uniqueness-sec-2} reveal that
\begin{align*}
 \E& [H(| \Delta_{T\wedge \tau_{R}\wedge S_{\delta_{0}}}|)]    - \E[H( |\Delta_{t_{0}}  |)] \\&=      \E\biggl[\int_{t_{0}}^{T\wedge \tau_{R}\wedge S_{\delta_{0}}} \frac{2 \lan \Delta_{s}, b(Y(s))- b(J^{n}(s))\ran + |\sigma(Y(s))|^{2}-4|\lan\sigma(Y(s)), \Delta_{s}\ran|^{2}}{(1+ |\Delta_{s}|^{2})^{2}} \d s\biggr]    \\
    &   \ \ + \E\biggl[\int_{t_{0}}^{T\wedge \tau_{R}\wedge S_{\delta_{0}}} \!\!\int_{U} \biggl(  H(|\Delta_{s} + c(Y(s-),u)|) - H(|\Delta_{s} |)- \frac{2\lan \Delta_{s}, c(Y(s-),u)\ran}{(1+ |\Delta_{s}|^{2})^{2}}\biggr)\nu(\d u) \d s\biggr] \\
    & \le   K_{R} \E\biggl[\int_{t_{0}}^{T\wedge \tau_{R}\wedge S_{\delta_{0}}} \biggl( \varrho(H(|\Delta_{s}|)) + |\sigma(Y(s))|^{2} + \int_{U} |c(Y(s-),u)|^{2} \nu(\d u)\biggr)\d s\biggr]\\
    & \le K_{R}\E\biggl[\int_{t_{0}}^{T\wedge \tau_{R}\wedge S_{\delta_{0}}}  \bigl( \varrho(H(|\Delta_{s}|)) + 1+ |Y(s)|^{2} \bigr) \d s\biggr]\\
    & \le K_{R}\E\biggl[\int_{t_{0}}^{T}  \bigl( \varrho(H(|\Delta_{s\wedge \tau_{R}\wedge S_{\delta_{0}}}|)) + 1+ |Y(s\wedge \tau_{R}\wedge S_{\delta_{0}})|^{2} \bigr) \d s\biggr],
\end{align*} where the second  last inequality follows from the linear growth condition given by Assumption \ref{assumption-non-linear-growth-sec-2}, and $K_{R}$  is a positive constant. Also, throughout the proof, $K_{R}$ is generic positive constant whose exact value may change from line to line.
 Furthermore, by virtue of \cite{ZhuYB-15}, we have $\E[\sup_{t \in [0, T]} |Y(t)|^{2}] \le K $, where $K$ is a positive constant independent of $t_{0}$ and $R$. Thus we have \begin{align*}
 \E & [H(|\Delta_{T\wedge  S_{\delta_{0}}\wedge \tau_{R}} |)]   \le  \E[H(| \Delta_{t_{0}} |)] + K_{R} (T-t_{0})  + K_{R} \int_{t_{0}}^{T }   \varrho(\E[H(|\Delta_{s\wedge S_{\delta_{0}}\wedge \tau_{R}}|)]) \d s.
\end{align*} Note that we also used Jensen's inequality to obtain the above inequality. Consequently as in the proof of Theorem \ref{thm-path-uniqueness-sec-2}, we have\begin{equation}\label{eq2-irreducibility-proof}
 \E[H(|\Delta_{T\wedge \tau_{R} \wedge S_{\delta_{0}}} |)]   \le  G^{-1} \bigl(G( \E[H(| \Delta_{t_{0}} |)] + K_{R} (T-t_{0})) + K_{R} (T-t_{0})\bigr),
\end{equation}where $G(r) : = \int_{1}^{r}\frac{\d \xi}{\varrho(\xi)}$ and $G^{-1}$ is the (left) inverse function of $G$: $G^{-1}(x) : = \inf\{y \ge 0: G(y) \ge x \}$, $x\in \R$.

Next we observe that for the positive constant $ \frac{1}{H(\delta_{0})} = 1+ \frac{1 }{\delta_{0}^{2}}  $, we have   \begin{equation}
\label{eq3-irreducibility-proof}
\E[H(|\Delta_{T}|)] \le  \frac{1}{H(\delta_{0})}  \E[H(|\Delta_{T\wedge  S_{\delta_{0}}} |)].
\end{equation}To see this, we notice that on the set $\{ S_{\delta_{0} }< T\wedge \tau_{R} \}$, we have $|\Delta_{T\wedge\tau_{R}\wedge S_{\delta_{0}}}| \ge \delta_{0}$ and hence $H(\delta_{0}) \le H(|\Delta_{T\wedge \tau_{R}\wedge  S_{\delta_{0}}} |)$ since $H$ is increasing. Therefore,
\begin{align*}\E[H(|\Delta_{T\wedge \tau_{R}\wedge  S_{\delta_{0}}} |)]& = \E[H(|\Delta_{T\wedge \tau_{R}}|)I_{\{T\wedge \tau_{R} \le S_{\delta_{0}} \}}] + \E[H(|\Delta_{S_{\delta_{0}}}|)I_{\{   S_{\delta_{0}} < T\wedge \tau_{R} \}}] \\ &\ge \E[H(|\Delta_{T\wedge \tau_{R}}|)I_{\{T\wedge \tau_{R} \le S_{\delta_{0}} \}}] + H(\delta_{0}) \P\{ S_{\delta_{0}} < T\wedge \tau_{R} \}. \end{align*}
Then it follows that \begin{align*}
  & \frac{ \E[H(|\Delta_{T\wedge\tau_{R} \wedge  S_{\delta_{0}}} |)] }{H(\delta_{0})} -  \E[H(|\Delta_{T\wedge\tau_{R} }|)]\\
   & \quad \ge  \frac{\E[H(|\Delta_{T\wedge\tau_{R} }|)I_{\{T\wedge\tau_{R}  \le S_{\delta_{0}} \}}] + H(\delta_{0}) \P\{ S_{\delta_{0}} < T\wedge\tau_{R}  \} }{H(\delta_{0})}    - \E[H(|\Delta_{T\wedge\tau_{R} }|)]  \\
    &  \quad \ge   \P\{ S_{\delta_{0}} < T\wedge\tau_{R}  \} +\E[H(|\Delta_{T\wedge\tau_{R} }|)I_{\{T\wedge\tau_{R}  \le S_{\delta_{0}} \}}]  - \E[H(|\Delta_{T\wedge\tau_{R} }|)] \\
    & \quad= \P\{ S_{\delta_{0}} < T\wedge\tau_{R}  \} - \E[H(|\Delta_{T\wedge\tau_{R} }|)I_{\{  S_{\delta_{0}} <  T\wedge\tau_{R}  \}}] \\
    &\quad \ge   \P\{ S_{\delta_{0}} < T\wedge\tau_{R}  \}  -\E[ 1\cdot I_{\{  S_{\delta_{0}} <  T\wedge\tau_{R}  \}}]  =0.
\end{align*} Consequently $\E[H(|\Delta_{T\wedge\tau_{R} }|)] \le  \frac{ \E[H(|\Delta_{T\wedge\tau_{R} \wedge  S_{\delta_{0}}} |)] }{H(\delta_{0})} $ for each $R > 0$. Thanks to  Theorem \ref{thm-no-explosion-sec-2}, $\lim_{R\to\infty}\tau_{R}  = \infty$ a.s.  Also note that $H$ is uniformly bounded. Thus, by the bounded convergence theorem, passing to the limit as $R\to \infty$ establishes \eqref{eq3-irreducibility-proof}.

For any $\e > 0$, we can choose some $R_{0} > 0 $ sufficiently large so that $\P\{\tau_{R_{0}} \le T \wedge S_{\delta_{0}} \} \le \P\{ \tau_{R_{0}} \le T\} < \e$. Then we have from \eqref{eq2-irreducibility-proof} that \begin{align}\label{eq4-irreducibility-proof}
 \nonumber \E[H(|\Delta_{T\wedge  S_{\delta_{0}}} |)]   &  =  \E[H(|\Delta_{T\wedge  S_{\delta_{0}}} |)I_{\{T \wedge S_{\delta_{0}} \le \tau_{R_{0}} \}}]  +  \E[H(|\Delta_{T\wedge  S_{\delta_{0}}} |)I_{\{T \wedge S_{\delta_{0}} >  \tau_{R_{0}} \}}]   \\
  \nonumber  &   \le  \E[H(|\Delta_{T\wedge\tau_{R_{0}}\wedge  S_{\delta_{0}}} |)I_{\{T \wedge S_{\delta_{0}} \le \tau_{R_{0}} \}}] +  \P\{\tau_{R_{0}} \le T \wedge S_{\delta_{0}} \}\\
    & \le  G^{-1} \bigl(G( \E[H(| \Delta_{t_{0}} |)] + K_{R_{0}} (T-t_{0})) + K_{R_{0}} (T-t_{0})\bigr) + \e.
\end{align}

The rest of the proof is very similar to those in the proof of Proposition 2.4 of \cite{Qiao-14}. Note that $Y$ satisfies the SDE
\begin{displaymath}
Y(t) : = x + \int_{0}^{t} [b(Y(s)) +   h^{n}(s)I_{\{s> t_{0}\}}] \d s +  \int_{0}^{t}\sigma(Y(s))\d W(s) + \int_{0}^{t}\int_{U} c(X(s-),u)\wdt N(\d s \d u),
\end{displaymath} for $t \in [0,T]$. Put $\wdt H(t): =I_{\{t> t_{0}\}}\sigma^{-1}(Y(t)) h^{n}(t)$ and $$M(t):= \exp\biggl\{ \int_{0}^{t}\lan \wdt H(s), \d W(s)\ran - \frac12\int_{0}^{t} |\wdt H(s)|^{2} \d s\biggr\}, \quad t \in [0, T].$$ As observed in \cite{Qiao-14}, $M$ is an a.s. strictly positive  martingale under $\P$ with $\E[M(T)]=1$,   the measure $\Q$ defined by $\Q(A) = \E[M(T) I_{A}], A \in \F_{T}$ is probability measure equivalent to $\P$ on $\F_{T}$,    $\wdt W(t): = W(t) + \int_{0}^{t}\wdt  H(s)\d s$ is a $\Q$-Brownian motion, and $\wdt N(\d t, \d u)$ is a $\Q$-compensated Poisson random measure with compensator $\d t\nu(\d u)$. Moreover, under $\Q$, $Y$ solves the SDE
\begin{displaymath}
Y(t) : = x + \int_{0}^{t} b(Y(s))   \d s +  \int_{0}^{t}\sigma(Y(s))\d \wdt W(s) + \int_{0}^{t}\int_{U} c(X(s-),u)\wdt N(\d s \d u), \ \ t\in [0, T].
\end{displaymath} By the pathwise uniqueness result established in Theorem \ref{thm-path-uniqueness-sec-2}, it follows that $\P\{|X^{x}(T) -a | > r\} = \Q\{|Y(T) -a| > r\}$.  Furthermore, since $\P, \Q$ are equivalent, the desired assertion $\P\{|X^{x}(T) -a| >  r \} < 1$ will follow if we can show that $\P\{ |Y(T) -a| > r\} < 1$. To this end, we deduce as follows. Since the function $H$ is increasing, we can use \eqref{eq3-irreducibility-proof}  and \eqref{eq4-irreducibility-proof}  to derive
\begin{align*}
       \P  \{ |Y(T) -a| > r\}  \le \P\{H( |Y(T) -a| ) > H(r)\} \le \frac{\E[H( |Y(T) -a| ) ]}{H(r)}  =   \frac{\E[H(  |\Delta_{T}| ) ]}{H(r)}\\
       \le \frac{ \E[H(|\Delta_{T\wedge  S_{\delta_{0}}} |)]}{H(r) H(\delta_{0})} \le \frac{ G^{-1} \bigl(G( \E[H(| \Delta_{t_{0}} |)] + K_{R_{0}} (T-t_{0})) + K_{R_{0}}(T-t_{0})\bigr) + \e }{H(r) H(\delta_{0})}.
\end{align*} Finally, in view of \eqref{eq1-irreducibility-proof} and the asymptotic properties  of $G$ and $G^{-1}$,  we can make the value of the last fraction in the above equation arbitrarily small by choosing  $n$ sufficiently large and $t_{0}$ sufficiently close to $T$. This completes the proof.
 \end{proof}

\begin{proof}[Proof of Lemma \ref{lem-nu-representation}]
We give a constructive proof motivated by \cite{Kurtz-11}.  Since $\nu(x, \cdot)$ is a $\sigma$-finite measure on $\R^{d}_{0}$, we can find a measurable partition $\{A_{n} \}_{n=-\infty}^{\infty}$ of $\R^{d}_{0}$ such that $0< \nu(x, A_{n}) \le 1$ for each $n$.
Now let $$\nu_{n}(x, \cdot) : = \nu (x, \cdot\cap A_{n}), \text{ and }\mu_{n} (x,  \cdot) : =  \frac{\nu_{n}(x, \cdot)}{\nu_{n}(x, \R^{d}_{0})}, \quad n \in \mathbb Z.$$   Obviously  we have $\nu(x, \Gamma)= \sum_{n=-\infty}^{\infty}\nu_{n}(x, \Gamma)$ for each  $\Gamma\in \B(\R^{d}_{0})$.
Using the measurable selection theorem (see, e.g. \cite{KuraR-65} or \cite[Chapter 12]{Stroock-V}), we can  choose $\nu_{n}(x, \cdot)$ so that $\nu_{n}(\cdot,\Gamma)$ is measurable for each $n$ and $\Gamma\in \B(\R^{d}_{0})$.
 For any complete and separable metric space $E$, denoting by $\mathcal P(E)$ the set of probability measures on $E$,  there exists a Borel measurable function $h: \mathcal P(E) \times [0,1] \mapsto \R^{d}$ such that $h(\mu,Z) \stackrel{d}{=} \mu$, where $\mu \in \mathcal P(E)$ and $Z$ is uniformly distributed on $[0,1]$.

 Now define functions $\gamma: \R^{d} \times \R\mapsto \R^{d}$ and $\lambda: \R^{d} \times \R\mapsto \R$ by $$\gamma(x, \xi) : =\sum_{k=-\infty}^{\infty} h(\mu_{k} (x, \cdot), \xi)I_{[k, k+1)}(\xi), \text{ and } \lambda(x, \xi) : =\sum_{k=-\infty}^{\infty} \nu_{k}(x, \R^{d}_{0})I_{[k, k+1)}(\xi).$$ Then it follows that for any $\Gamma \in \B(\R^{d}_{0})$, we  have
 \begin{align*}
 \int_{\R}    \lambda(x, \xi) I_{\Gamma}(\gamma(x, \xi)) \d\xi
 & = \sum_{k=-\infty}^{\infty} \int_{k}^{k+1} \nu_{k}(x, \R^{d}_{0}) I_{\Gamma}(h(\mu_{k} (x, \cdot),\xi))\d \xi \\
 & =  \sum_{k=-\infty}^{\infty}\nu_{k}(x, \R^{d}_{0})  \int_{k}^{k+1} I_{\Gamma}(h(\mu_{k} (x, \cdot),\xi))\d \xi  \\
  & =  \sum_{k=-\infty}^{\infty}\nu_{k}(x, \R^{d}_{0}) \mu_{k} (x, \Gamma)
  =  \sum_{k=-\infty}^{\infty}\nu_{k}(x, \Gamma) = \nu(x, \Gamma).
\end{align*}


   Since $0 \notin \Gamma$, we can write
\begin{align*}
 \nu(x, \Gamma ) & = \int_{\R}   \lambda(x, \xi) I_{\Gamma}(\gamma(x, \xi))  \d \xi
   = \int_{\R} \int_{0}^{1} I_{[0,\lambda(x, \xi)]}(\eta) \d \eta \, I_{\Gamma}(\gamma(x, \xi))   \d \xi\\
  & =  \int_{\R\times [0,1]} I_{\Gamma}(\gamma(x, \xi) I_{[0,\lambda(x, \xi)]}(\eta) ) \d \eta \d \xi.
\end{align*}
This  gives \eqref{eq-nu-measure-representation} with  
$   c (x, u) = \gamma(x, \xi) I_{[0,\lambda(x, \xi)]}(\eta),$
 $(U,\mathfrak U)= (\R\times [0,1], \B(\R\times [0,1]))$,  and $M(\cdot) $ being  the Lebesgue measure on $\R\times [0,1]$.
 The lemma is therefore proved.
 \end{proof}

\section*{\bf\large Acknowledgements}
 We would like to thank the anonymous reviewer for pointing out an important  reference to us and for his/her useful comments.
 The research was supported in
 part by the National Natural Science Foundation of China under Grant No. 11671034, the Beijing Natural Science Foundation under Grant No. 1172001, and the Simons foundation collaboration grant 523736.



\end{document}